\documentclass[12pt,reqno]{amsart}
\usepackage{float}
\usepackage[left=3cm,top=2cm,right=3cm,bottom=2cm]{geometry}
\usepackage{amsmath,amsfonts}
\usepackage{amssymb,amsthm,amscd,indentfirst}
\usepackage{graphicx}
\usepackage{graphics}
\usepackage[latin1]{inputenc}
\usepackage[T1]{fontenc}
\usepackage[english]{babel}
\usepackage{geometry}
\usepackage{tikz}
\usetikzlibrary{decorations.pathreplacing}
\usepackage{arydshln}
\usepackage{epsfig}
\usepackage[all]{xy}
\usepackage{hyperref}
\usetikzlibrary{backgrounds,arrows,shapes,trees} 
\title[Applications of rational difference equations to spectral graph theory]{Applications of rational difference equations to spectral graph theory: expanded version}
\author[E. R. Oliveira]{Elismar R. Oliveira}
\address{UFRGS - Universidade Federal do Rio Grande do Sul, Instituto de Matem\'atica e Estat\'stica, Porto Alegre, Brazil}\email{\tt elismar.oliveira@ufrgs.br}
\author[V. Trevisan]{Vilmar Trevisan}
\address{UFRGS - Universidade Federal do Rio Grande do Sul, Instituto de Matem\'atica e Estat\'stica, Porto Alegre, Brazil and \\
Department of Mathematics and Applications, University of Naples Federico II, Italy}\email{\tt trevisan@mat.ufrgs.br}

\newtheorem{obs}{Remark}
\newtheorem{exemplo}{Example}

\newtheorem{thm}{Theorem}
\newtheorem{defin}{Definition}
\newtheorem{prop}{Proposition}
\newtheorem{lema}{Lemma}
\newtheorem{coro}{Corollary}

\newcommand{\ero}[1]{\textcolor{blue}{#1}}

\begin{document}
\begin{abstract}

We study a general class of recurrence relations that appear in the application of a matrix diagonalization procedure. We find a general closed formula and determine the analytical properties of the solutions. We finally apply these findings in several problems involving eigenvalues of graphs.

\\[0.15cm]

\end{abstract}

\maketitle

\section{Introduction}

\ero{The expanded version presents some additional computations that were not suitable for the published version and small typo corrections. Those are highlighted in blue for sake of comparison.}

The main goal of this paper is to consider a general class of rational type recurrences appearing in graph applications, mainly in eigenvalue location, as a unified elementary form,
\begin{equation}\label{formphi}
  x_{j+1}= \varphi(x_{j}), j \geq 1
\end{equation}
where $\varphi(t)= \alpha + \frac{\gamma}{t}$, for $t\neq 0$,   $\alpha, \gamma \in \mathbb{R}$ are fixed numbers ($\gamma \neq 0$) and $x_{1}$ is a given initial condition. The explicit solution is a function $j \to f(j)$ such $x_{j}= f(j)$ for $j \geq 1$.

Perhaps this study is interesting \emph{per se}, but we explain now the motivation behind these recurrences that may be used for diagonalizing certain symmetric matrices and their important applications in graph theory.

For a real $ n \times n$  symmetric matrix $M=[m_{i,j}]$, we consider the \emph{graph of $M$}, as the graph with $n$ vertices $1,2,\ldots,n$ where an edge between $i$ and $j$ exists if and only if $m_{i,j} \not = 0$. We may see  $m_{i,j}$ as the weight of the edge $ij$, whereas $m_{i,i}$, the diagonal element, as the weight of the vertex $i$. Diagonalizing the matrix $M$ has important applications in numerical linear algebra in general. Our focus here are applications we will describe later on in spectral graph theory. In particular, if the graph of $M$ is a tree $T$, the diagonalization algorithm of Jacobs-Trevisan~\cite{JT2011} (J-T algorithm, for short) provides an efficient and ingenious procedure that can be executed on the tree itself. We provide the J-T algorithm in Figure~\ref{treealgo} for easy reference.

If we start the algorithm with a value $\alpha$, the algorithm computes a diagonal matrix $D_\alpha$ that is congruent to $M-\alpha I$, where the diagonal values are stored as the weight of the vertices. By the Silvester Law of Inertia, the inertia of $D_\alpha$ and $M-\alpha I$ are the same. This tell us that the number of (positive/negative/zero) final values of the vertices indicates the number of eigenvalue of $M$ that are (greater than/smaller than/equal to)  $\alpha$. We call this an \emph{eigenvalue location algorithm}.

The J-T algorithm works as follows. For a given $\alpha \in \mathbb{R}$, we initialize the weights of the vertices as $a(v_i) = m_{i,i} - \alpha$. For a child  $v_{j}$ of $v_{i}$, if the value of the diagonal entry is $a(v_{j}) \not = 0$ and the correspondent value in the positions $(i,j)$ and $(j,i)$ is $w$ then we can annihilate the nondiagonal elements of row and column $j$ using  $a(v_{j})$ while the value of the diagonal entry $a(v_{i})$ is replaced by
 $$a(v_{i})=\alpha_{i} - \frac{w^2}{a(v_{j})} $$ where $\alpha_{i}= m_{i,i}$ is the diagonal value in the position $(i,i)$. When the value $a(v_{j}) = 0$, the algorithm does some other procedure (see Figure~\ref{treealgo}).

\begin{figure}[h]
\centering
{\tt\begin{tabular}{p{14cm}}
\hline
Input: scalar $\alpha$; matrix $M$ of a tree $T$ with vertices $v_1, v_2,\ldots, v_n$, in postorder \\
Output: \, diagonal matrix $\Lambda$ congruent t $M + \alpha I$  \\
\hline
\vspace*{.1cm}
\noindent \textbf{Inicialize} $a(v_{i}) := m_{i,i}  - \alpha$, for all $v_i$ of $T$. \\
\textbf{for} $k=n$ \textbf{to} $1$  \textbf{do} \\
$\> \> \> \>  $  \textbf{if} $v_k$ is not a leaf, \textbf{then} \\
$\> \> \> \> \> \> \> \> \> \> \> $ \textbf{1.~ if }$a(v_i) \neq 0$ for all children $v_i$ of  $v_k$, \textbf{then} \\
$\> \> \> \> \> \> \> \> \> \> \> \> \> \> \> \> \> \> \> \> \> $  $a(v_k) \leftarrow a(v_k) - \displaystyle\sum_{v_{i}} {\frac{(m_{i k} )^2}{a(v_i) }}.$ \\
$\> \> \> \> \> \> \> \> \> \> \> $ \textbf{2.~ if } $a(v_i) =0$ for some children $v_i$ of $v_k$, \textbf{then }\\
$\> \> \> \> \> \> \> \> \> \> \> \> \> \> \> \> \> \> \> \>\> $ choose a vertex $v_j$ that $a(v_j) = 0;$ \\
$\> \> \> \> \> \> \> \> \> \> \> \> \> \> \>\> \> \> \> \>\> $ $a(v_k) \leftarrow - \displaystyle\frac{(m_{j k})^{2}}{2}; \,\,\,\, a(v_j) \leftarrow 2;$ \\
$\> \> \> \> \> \> \> \> \> \> \> \> \> \> \> \> \> \> \> \>\> $  \textbf{if} $v_k$ has a parent $v_{\ell}$, \textbf{then} remove the edge $\{v_k, v_{\ell}\}$. \\
\vspace*{.1cm}\\
\hline
\end{tabular}
}
\vspace{0.2cm}
\caption{The J-T algorithm, $\textit{Diagonalize}(T, \alpha)$.}\label{treealgo}
\end{figure}

The information provided by the signs of the final values turns out to be of paramount importance for several applications in spectral graph theory. To be more concrete, and understand why the sequences appear, we consider three cases that will be prototypical to our studies.

\begin{itemize}
  \item Index or spectral radius of a tree $T$. This is the largest eigenvalue of the adjacency matrix of $T$, let us denote it by $\lambda$. We consider applying the J-T algorithm in a pendant path, $M$ is the adjacency matrix of $T$ and we initialize the diagonal with $0-\lambda$ ($M_{i,i}=0$ and $\alpha=\lambda$). The nonzero entries are equal to 1. Starting in the end vertex of the path one obtains $z_1=-\lambda$, $z_2=-\lambda - \frac{1^2}{z_1}=-\lambda - \frac{1}{z_1}$ and so on. Therefore, we have a recursion
      \begin{equation}\label{recurs index}
        \left\{
        \begin{array}{ll}
          z_1=-\lambda \\
          z_{j+1}=-\lambda - \frac{1}{z_{j}}, \; j \geq 1.
        \end{array}
      \right.
      \end{equation}
      A careful analysis of this recurrence relation (with different initial conditions), has been used to compare indices in classes of trees, enabling one to solve combinatorial/algebraic problems where traditional techniques had failed (see, for example \cite{Belardo2019,oliveira2020spectral,OlivStevTrev}).
  \item Average of the Laplacian eigenvalues of a tree $T$. We initialize the diagonal with ${\rm degree}(v_i)-d$ where $d=2-\frac{2}{n}$ ($M_{i,i}={\rm degree}(v_i)$ and $\alpha=d$ is the average of the Laplacian eigenvalues). The correspondent entries are equal to -1. Starting in the end vertex of a path one obtains $a_1=1-d=-1+\frac{2}{n}$, $a_2=2-d - \frac{(-1)^2}{a_1}=\frac{2}{n} - \frac{1}{a_1}$ and so on. Therefore, we have a recursion
      \begin{equation}\label{recurs Laplacian average}
        \left\{
        \begin{array}{ll}
          a_1=-1+\frac{2}{n} \\
          a_{j+1}=\frac{2}{n} - \frac{1}{a_{j}}, \; j \geq 1.
        \end{array}
      \right.
      \end{equation}
      This recurrence relation, used in an ingenious way, was the main tool to prove that at least half of the Laplacian eigenvalues of a tree are smaller than its average ~\cite{Jacobs2021}.
  \item  An eigenvalue $\lambda \in [0,\;2]$ of the normalized Laplacian of a tree $T$. We initialize the diagonal with $1-\lambda$ ($M_{i,i}=1$ and $\alpha=\lambda$). The correspondent entries are equal to $\frac{-1}{\sqrt{deg(v_i)deg(v_j)}}$ which is always equal to $\frac{-1}{\sqrt{2}}$ in the end of a path and $\frac{-1}{\sqrt{2\cdot 2}}=-\frac{1}{2}$ for the next vertices. Starting in the end vertex of a path one obtain $x_0=1-\lambda$, $x_1=1-\lambda- \frac{(\frac{-1}{\sqrt{2}})^2}{x_0}=1-\lambda- \frac{1/2}{x_0}$, $x_2=1-\lambda- \frac{(\frac{-1}{2})^2}{x_1}=1-\lambda- \frac{1/4}{x_1}$, $x_3=1-\lambda- \frac{(\frac{-1}{2})^2}{x_2}=1-\lambda- \frac{1/4}{x_2}$ and so on. Therefore, we have a recursion
      \begin{equation}\label{recurs norm Laplacian average}
        \left\{
        \begin{array}{ll}
          x_0=1-\lambda \\
          x_1=1-\lambda- \frac{1}{2(1-\lambda)}\\
          x_{j+1}=1-\lambda - \frac{1/4}{x_{j}}, \; j \geq 1.
        \end{array}
      \right.
      \end{equation}
\end{itemize}

As these examples show, applying the J-T algorithm on a path of a tree, produces certain numerical rational sequences. They are all of the general form given by equation \eqref{formphi}.

The main purpose of this paper is to understand the analytical behaviour of these sequences $x_j$ when $j$ is seen as a continuous variable, which allows one to obtain information about the sign on the discrete values of $j$. The passage from natural numbers, representing the vertices of a graph, to arbitrary real numbers allows one  to find solutions of analytical equations which will define intervals of indices where a certain property is true. By  applying the general results to particular cases, we also demonstrate the potential of this technique in applications to spectral graph theory. For example, we determine limit points of spectral radius of certain trees and study and analyse the dependence of the solutions and the initial conditions on a given parameter, applying to study the number of eigenvalues smaller than its average.

In some cases, the applications we made require to study the dependence of the solution with respect to a given parameter $a$ that may appear both in the initial condition and/or in the recurrence formula.

Another aspect that we investigate and use in applications, is the extension properties. Here we mean to extend the integer variable $j$ to values in $\mathbb{R}$ and to use the structure of the resulting function to derive properties of the original sequence.

In several cases, the dynamical behaviour of the function $\varphi(t)= \alpha + \frac{\gamma}{t}$ brings us a lot of information about the correspondent sequence. The knowledge of these properties allows one to study some important problems in spectral graph theory regarding eigenvalue location.

The paper is organized as follows.
Next section presents a reduction method allowing us to transform the general recurrence given by \eqref{formphi} into a second order linear recurrence, which in turn, leads to a closed formula for the solution. In Section \ref{sec:prop}, we present general properties of solutions given in Section \ref{sec:red}. The next two sections present new applications in spectral graph theory. Section \ref{sec:limit} deals with limit points of spectral radius of graphs and uses the solutions to prove two known results. The rational is to show the potential use of these techniques in this area. Section \ref{sec:param} studies the dependence of the solutions and the initial conditions on a given parameter. We also use a particular example - study the number of eigenvalues smaller than its average - to show how an analytical approach may be used to obtain very precise results. We finalize the paper with some concluding remarks and open problems.

\section{The reduction method}\label{sec:red}

Before reducing equation $x_{j+1}= \varphi(x_{j})$, where   $\varphi(t)= \alpha + \frac{\gamma}{t},\; \alpha, \gamma \in \mathbb{R}$, into a second-order linear recursion, we will discuss briefly two general properties. \ero{We recall that a solution is a necessarily infinite sequence of values $(x_{j})_{j\in \mathbb{N}}$. Since we are dealing with a rational recurrence we must require that the initial condition $x_1$ be such that $x_{j}\neq 0$ for all $j \geq 1$. Thus, when we are talking about solutions of the rational recurrence we are implicitly assuming that the initial condition satisfy $x_1 \not\in \{0\} \cup \mathcal{Z}$, where
\[\mathcal{Z}:=\{y \neq 0\, | \,   \varphi^k(y) \neq 0 , \forall k \geq 1\} = \bigcup_{k\geq 1} \varphi^{-k}(0).\]
\begin{defin}
	The set $\mathcal{Z}$ is called the null set for the rational recurrence $x_{j+1}= \varphi(x_{j})$.
\end{defin}}

\ero{We notice that it is possible to explore the  symmetry of the formula $\varphi(t)= \alpha + \frac{\gamma}{t}$
	\[\varphi(-t)= \alpha + \frac{\gamma}{-t}= -(-\alpha) - \frac{\gamma}{t}= -\left(-\alpha + \frac{\gamma}{t} \right),\]
to reduce the number of dynamically distinct cases.  This means that despite the fact that the recurrence are different  its solutions can be recovered by the symmetry around the origin.
\begin{lema}\label{lem:symmetry}
	 The sequence $(x_j)_{j \in \mathbb{N}}$ is a solution of  $x_1=a$ and  $x_{j+1}=\varphi(x_j) $, where   $\varphi(t)=\alpha +\frac{\gamma}{t}, \; t\neq 0$, if and only if,   $(y_j)_{j \in \mathbb{N}}$ is a solution of $y_1=-a$ and  $y_{j+1}=\psi(y_j) $, where   $\psi(t)=-\alpha +\frac{\gamma}{t}, \; t\neq 0$.
\end{lema}
\begin{proof}
Indeed,
\[x_{j+1}=\alpha +\frac{\gamma}{x_{j}} \iff -x_{j+1}=-\alpha -\frac{\gamma}{x_{j}}  \iff -x_{j+1}=-\alpha +\frac{\gamma}{-x_{j}}, \]
so if we take $y_{j}=-x_{j}$ we obtain the equivalence because $y_{1}=-x_{1}=-a$.
\end{proof}}

\subsection{Zeroes of $\varphi$}
The first problem we face is to find the zeroes of $\varphi$. If  $x_{j}=0$ for some $j$, then we can not compute $x_{j+1}$. \ero{We observe that in the applications to the J-T algorithm to locate a real value $\lambda$(for example $\alpha=-\lambda$, for the adjacency matrix), this means that the initial value $\lambda$ could belong to the spectrum of the associated matrix depending on the number of zeroes we find when processing the children of a vertex to be more than one or, if the graph is just a path and we are processing the root vertex. Although,  in any case there is a procedure to choose the next term in the algorithm, which is not  by computing $x_{j+1}=\varphi(x_{j})$, see Figure~\ref{treealgo} for the exact attribution.} Therefore, the infinite orbit $(x_{j})_{j\geq 1}$ are the complement of the pre-images of  zero by $\varphi$. Hence, $(x_j)$ is finite if, and only if, $x_1= \varphi^{-n}(0)$ for some $n \in \mathbb{N}$. We recall that $\psi(t)= \frac{\gamma}{t - \alpha}$, for $t\neq \alpha$, is the inverse of $\varphi(t)= \alpha + \frac{\gamma}{t}$. Therefore the initial conditions with finite orbit are the points $\{\psi^{n}(0)\; |\; n\geq 1\}$.

Evaluating $\psi(t)$ we obtain \ero{the null set}
$$ \ero{\mathcal{Z}=}\left\{\psi(0)=-\frac{\gamma}{\alpha },\; \psi^{2}(0)=-\frac{\gamma}{\frac{\gamma}{\alpha }+\alpha }, \; \ldots\right\}.$$

\begin{exemplo}\label{ex:zeros} Considering $\alpha=2$ and $\gamma=-1$, we have $\varphi(t)= 2 - \frac{1}{t}$ and $\psi(t)= \frac{1}{2- t}$. It is easy to see that the sequence $\left\{\psi(0),\; \psi^{2}(0), \; ...\right\}$ is the sequence $\ero{\mathcal{Z}=(\frac{n-1}{n})_{n \geq 2}}$.  Thus, if we take $x_1=3/4$ (that is, $n=4$) we get $x_2=2/3$, $x_3=1/2$ and $x_4=0$, so we can not compute $x_5$. This is a powerful information: as $\frac{n-1}{n} \to 1$ increasingly, for any initial condition $x_1$ outside of a neighborhood of $1$ we can iterate $\varphi$ to any order, except for a finite set of numbers of the form $x_1=\frac{n-1}{n}, n\leq n_0$.
\end{exemplo}

\subsection{Fixed points of $\varphi$}
Regarding the infinite orbits $(x_{j})_{j\geq 1}$ the relevant question are about the accumulation points, that is, the asymptotic behaviour of $\varphi$. As $\varphi$ is piecewise monotonous, we notice that such a limit, if it exists, must be some fixed point of $\varphi$.

Evaluating $\varphi(t)= \alpha + \frac{\gamma}{t}=t$ we see that the only possible zeroes are the solutions of
\begin{equation}\label{fixed points phi}
  t^2 -\alpha t -\gamma  =0.
\end{equation}
which is the characteristic equation of the auxiliary recursion $c_{j}$ (see \eqref{reduced recursion} below).

Therefore the solution is
\begin{equation}\label{eq: fixed point varphi t}
  t= \frac{\alpha}{2} \pm \frac{1}{2}\sqrt{\alpha^2 + 4 \gamma}.
\end{equation}

In the analysis of the equivalent linear second order recurrence we make below, where the characteristic polynomial appears, the fixed point are important.

\subsection{Reduction}
Dealing with the non-linear recursion \eqref{formphi} one can apply a reduction method to transform $x_{j+1}= \varphi(x_{j})$ into a linear equation. Given the nature of the rational functions, we propose to consider an auxiliary sequence $c_{j} \neq 0$\ero{(we can do that because $0\neq x_1 \not\in  \mathcal{Z}$)} and try to find $x_{j}$ in the form
\begin{equation}\label{reduction seq}
  x_{j}=\frac{c_{j+1}}{c_{j}}, \; j \geq 1.
\end{equation}

Substituting \eqref{reduction seq} in \eqref{formphi} we obtain
\begin{equation}\label{reduced recursion}
  c_{j+2} -\alpha c_{j+1} -\gamma c_{j} =0.
\end{equation}
\color{black}
\subsection{Analysis of a second-order linear recursion}
For basic results on difference equations we refer the book \cite{elaydi2005introduction}, but we recall here some facts about a second-order linear recursion such as \eqref{reduced recursion},
\begin{equation}\label{general second-order linear recursion}
   c_{j+2} + U c_{j+1} + V c_{j} =0
\end{equation}
where $U,V \in \mathbb{R}$ are fixed numbers. Additionally, we assume that \eqref{general second-order linear recursion} is irreducible, that is, $V\neq0$. We assign to that the characteristic equation $\theta^2 + U \theta + V  =0$ which has three possible solutions:
\begin{itemize}
  \item[a)] Only one real root $\theta$. In this case, the solution of (general second-order linear recursion) is
  \begin{equation}\label{sol one root}
    c_{j}=(A +j B) \theta^{j}
  \end{equation}
  \item[b)] Two real roots $\theta, \theta'$. In this case, the solution of (general second-order linear recursion) is
  \begin{equation}\label{sol two roots}
    c_{j}=A \theta^{j} + B (\theta')^{j}
  \end{equation}
  \item[c)] Two conjugated complex roots $\rho \; e^{\pm i\phi}$. In this case the solution of (general second-order linear recursion) is
  \begin{equation}\label{sol two complex roots}
    c_{j}=\rho^{j} (A \cos(j\phi) + B \sin(j\phi) )
  \end{equation}
\end{itemize}

Denoting $\Delta=\alpha^2 + 4 \gamma$ we can find an explicit formula in each case.\\

Case 1: $\Delta=0$\\
In this case the only root is $\theta= \frac{\alpha}{2}$. Using \eqref{sol one root} we obtain $c_{j}=(A +j B) \theta^{j}$. If $B=0$ then $c_{j}=A \theta^{j}$ and $x_{j}=\frac{c_{j+1}}{c_{j}}=\theta$ is the trivial solution, which is obvious because  $\varphi(\theta)=\theta$. Thus we can suppose $B\neq 0$ and obtain
$$x_{j}=\frac{c_{j+1}}{c_{j}}= \frac{(A +(j+1) B) \theta^{j+1}}{(A +j B) \theta^{j}}=\theta \frac{(A +j B) +B}{(A +j B)}= \theta\left(1+ \frac{1}{(A/B +j )}\right).$$
As $B\neq 0$ we obtain the formula for non trivial solutions
\begin{equation}\label{explicity one root case}
   x_{j}=\theta\, \left(1+ \frac{1}{(\beta +j )}\right)
\end{equation}
where $\beta:=\left(\frac{2\theta-x_1}{x_1-\theta}\right) \in \mathbb{R}$ is defined by the initial point $x_1$.\\

Case 2: $\Delta>0$\\ In this case the two real roots are $\theta= \frac{\alpha}{2} - \frac{1}{2}\sqrt{\alpha^2 + 4 \gamma}$ and $\theta'= \frac{\alpha}{2}+ \frac{1}{2}\sqrt{\alpha^2 + 4 \gamma}$ \ero{(we will always use the negative sign for $\theta$ regardless  $\theta< \theta'$ or  $\theta>\theta'$). Notice that $\theta+\theta'=\alpha$ and $\theta \,\theta'=-\gamma$.} Using \eqref{sol two roots} we obtain
$c_{j}=A \theta^{j} + B (\theta')^{j}$ and $$x_{j}=\frac{c_{j+1}}{c_{j}}= \frac{A \theta^{j+1} + B (\theta')^{j+1}}{A \theta^{j} + B (\theta')^{j}}.$$
As $\gamma \neq 0$ we know that $\theta \neq 0$ and $\theta' \neq 0$. Also, $A$ and $B$ can not be simultaneously zero. If $B=0$ then $c_{j}=A \theta^{j}$ and $x_{j}=\frac{c_{j+1}}{c_{j}}=\theta$ is the trivial solution, which is obvious because  $\varphi(\theta)=\theta$. Thus we can suppose $B\neq 0$ and divide the formula by  $B (\theta')^{j}$ obtaining
$$x_{j}=\frac{A/B \;\theta \;\left(\frac{\theta}{\theta'}\right)^{j} + \theta'}{A/B \left(\frac{\theta}{\theta'}\right)^{j} + 1}=\frac{\theta \beta \left(\frac{\theta}{\theta'}\right)^{j} + \theta'}{\beta \left(\frac{\theta}{\theta'}\right)^{j} + 1}=\frac{\theta \left(\beta \left(\frac{\theta}{\theta'}\right)^{j}+1 -1 \right) + \theta'}{\beta \left(\frac{\theta}{\theta'}\right)^{j} + 1}.$$
Therefore the explicit equation  for other than the trivial solution $x_{j}=\theta$ is
\begin{equation}\label{explicity two roots case}
   x_{j}=\theta + \frac{\theta'-\theta}{\beta \left(\frac{\theta}{\theta'}\right)^j + 1}
\end{equation}
where  \ero{$\beta:=\left(\frac{\theta'}{\theta}\right)\left(\frac{\theta'-\theta}{x_1-\theta} -1\right) \in \mathbb{R}$ }is well defined by the initial point $x_1\neq \theta$. Notice that the other trivial solution $x_{j}=\theta'$ is obtained from $\beta=0$ (or equivalently $A=0$).\\

Case 3: $\Delta<0$\\ Since $\alpha^2 + 4 \gamma<0$  we obtain $\gamma < -\frac{1}{4}\alpha^2 \leq 0$ thus $\varphi(\theta)=\theta$ has no solution. In the case $\alpha \neq 0$ we have two complex roots are $Z= \frac{\alpha}{2} + \frac{i}{2}\sqrt{-\alpha^2 - 4 \gamma}$ and $\bar{Z}=\frac{\alpha}{2}- \frac{i}{2}\sqrt{-\alpha^2 - 4 \gamma}$. Using \eqref{sol two complex roots} we obtain
$ c_{j}=\rho^{j} (A \cos(j\phi) + B \sin(j\phi) )$ where
$$\rho=\sqrt{\left(\frac{\alpha}{2}\right)^2+\left(\frac{1}{2}\sqrt{-\alpha^2 - 4 \gamma}\right)^2}=\sqrt{-\gamma},$$
$$\phi=\arctan\left(\frac{\sqrt{-\alpha^2 - 4 \gamma}}{\alpha}\right) \text{ if } \alpha>0,$$
$$\phi=\arctan\left(\frac{\sqrt{-\alpha^2 - 4 \gamma}}{\alpha}\right)+\pi \text{ if } \alpha<0$$
considering the branch $\left(-\frac{\pi}{2}, \frac{\pi}{2}\right)$ of the function $\tan$, and
$$x_{j}=\frac{c_{j+1}}{c_{j}}= \frac{\rho^{j+1} (A \cos((j+1)\phi) + B \sin((j+1)\phi) )}{\rho^{j} (A \cos(j\phi) + B \sin(j\phi) )}=$$
$$ = \rho\,\frac{\frac{A}{\sqrt{A^2+B^2}} \cos((j+1)\phi) + \frac{B}{\sqrt{A^2+B^2}}  \sin((j+1)\phi) }{\frac{A}{\sqrt{A^2+B^2}}  \cos(j\phi) + \frac{B}{\sqrt{A^2+B^2}}  \sin(j\phi) }.$$
We denote $\omega \in [0,\,2\pi)$ the angle such that $\left(\frac{A}{\sqrt{A^2+B^2}}, \; -\frac{B}{\sqrt{A^2+B^2}}\right)=(\cos(\omega), \; \sin(\omega))$. Using the addition formula we get
$$\frac{A}{\sqrt{A^2+B^2}} \cos((j+1)\phi) + \frac{B}{\sqrt{A^2+B^2}}  \sin((j+1)\phi)= \cos((j+1)\phi +\omega)$$
and
$$\frac{A}{\sqrt{A^2+B^2}}  \cos(j\phi) + \frac{B}{\sqrt{A^2+B^2}}  \sin(j\phi) =\cos(j\phi +\omega)$$
thus
$$x_{j}=\rho\,\frac{\cos((j+1)\phi +\omega) }{\cos(j\phi +\omega) }=\rho\,\frac{\cos((j\phi+\omega) +\phi) }{\cos(j\phi +\omega)}=$$
$$=\rho\,\frac{\cos(j\phi+\omega )\cos(\phi)-\sin(\phi)\sin(j\phi+\omega) }{\cos(j\phi +\omega)}.$$
Therefore the explicit solution for $\alpha\neq 0$ is
\begin{equation}\label{explicity two complex roots case}
   x_{j}=\rho\, \left(\cos(\phi)  - \sin(\phi) \tan(j\phi +\omega) \right)
\end{equation}
where  $\omega:=-\phi+\arctan\left(\frac{\cos(\phi) - (x_1/\rho)}{\sin(\phi)}\right) \in [0,\,2\pi)$ is defined by the initial point $x_1$. The case $\alpha=0$ does not appear in the graph applications and it is, in a certain way, trivial because $c_{j+2} -\alpha c_{j+1} -\gamma c_{j} =0$ became $c_{j+2} =\gamma c_{j}$, a uncoupled equation. Taking $c_1=1$ and $c_2=x_1$ ($x_{1}=\frac{c_{2}}{c_{1}}$) we get $c_{2k} =\gamma^{k-1} c_{2}= \gamma^{k-1}  x_1$ and $c_{2k+1} =\gamma^{k} c_{1}= \gamma^{k}$ then
$x_{j}=\frac{c_{j+1}}{c_{j}}=
\left\{
  \begin{array}{ll}
    \frac{\gamma}{x_1}, & j=2k \\
    x_1, & j=2k+1
  \end{array}
\right.
, \; j \geq 1$. This is obviously the solution of $x_{j+1}=\frac{\gamma}{x_{j}}$ for a given $x_1 \neq 0$.\\

Summarizing, we have proved the theorem below, which is well known in a general setting but is useful to have the explicit computations of the general formulas, given a prescribed initial condition, as we will see on the applications.
\begin{thm} \label{sol general rational difference equation}
Consider the rational first order difference equation $x_{j+1}=
\varphi(x_{j}), j \geq 1$, where $\varphi(t)= \alpha + \frac{\gamma}{t}$, for
$t\neq 0$,   $\alpha, \gamma \in \mathbb{R}$ are fixed numbers ($\gamma \neq
0$) and \ero{$0\neq x_1 \not\in  \mathcal{Z}$} is a given initial condition. Then the general solution is
one of three possibilities according the sign of $\Delta=\alpha^2 + 4
\gamma$:
\begin{itemize}
  \item[Type 1:] For $\Delta=0$ the solution is
$$ x_{j}=\theta\, \left(1+ \frac{1}{(\beta +j )}\right),$$
where $\theta= \frac{\alpha}{2}$ and $\beta \in \mathbb{R}$ is defined by the initial point $x_1 \neq \theta$ by the formula $\beta=-1+\frac{\theta}{x_1-\theta}$. If $x_1 =\theta$ then $x_j=\theta, \, \forall j \geq 1$ is the solution.
  \item[Type 2:] For $\Delta>0$ the solution is
$$x_{j}=\theta + \frac{\theta'-\theta}{\beta \left(\frac{\theta}{\theta'}\right)^j + 1},$$
where \ero{$\theta= \frac{\alpha}{2} - \frac{1}{2}\sqrt{\alpha^2 + 4 \gamma}$,  $\theta'= \frac{\alpha}{2}+ \frac{1}{2}\sqrt{\alpha^2 + 4 \gamma}$ } and $\beta \in \mathbb{R}$ is defined by the initial point $x_1 \neq \theta$ by the formula $\beta=\frac{\theta'}{\theta}\,\left(\frac{\theta'-\theta}{x_1-\theta}-1\right)$. If $x_1 =\theta$ then $x_j=\theta, \, \forall j \geq 1$ is the solution.
  \item[Type 3:] For $\Delta<0$ the solution is
$$x_{j}= \rho\, \left(\cos(\phi)  - \sin(\phi) \tan(j\phi +\omega) \right),$$
where $\rho=\sqrt{-\gamma},$ $\phi=\arctan\left(\frac{\sqrt{-\alpha^2 - 4 \gamma}}{\alpha}\right) \text{ if } \alpha>0$,
$\phi=\arctan\left(\frac{\sqrt{-\alpha^2 - 4 \gamma}}{\alpha}\right)+\pi \text{ if } \alpha<0$ and  $\omega \in [0,\,2\pi)$ is defined by the initial point $x_1$ by the formula $\omega= -\phi +\arctan\left(\cot(\phi)-\frac{x_1}{\rho}\csc(\phi)\right)$.
If $\alpha=0$  then
$x_{j}=
\left\{
  \begin{array}{ll}
    \frac{\gamma}{x_1}, & j=2k \\
    x_1, & j=2k+1
  \end{array}
\right.
,$ $ j \geq 1$ is the solution of $x_{j+1}=\frac{\gamma}{x_{j}}$ for a given $x_1 \neq 0$.
\end{itemize}
\end{thm}

We analyse now a few properties of the solutions obtained that are useful for the applications we have in mind.

\section{General properties of the solutions}\label{sec:prop}
By \eqref{formphi} we know that the rational recursion $ x_{j+1}= \varphi(x_{j}), \; j \geq 1$ is defined by the rational map $\varphi(t)= \alpha + \frac{\gamma}{t}$, for $t\neq 0$. We assumed that $\gamma \neq 0$ \ero{(otherwise $x_j=\alpha, \forall j$ will be a trivial solution) and the initial point must be $0\neq x_{1} \not\in  \mathcal{Z}$.} We now give some properties of the solutions, some of them analytical.

\subsection{Reversing initial conditions}
In some cases we do not know what is $x_1$ but we know the value $x_{r}$ for some $r>1$. Usually we can not reverse the role of the recursion  $x_{j+1}= \varphi(x_{j})$ but it is not the case for $\varphi$.
\begin{lema}\label{lema reverse}
  Let $\psi$ be the inverse function of $\varphi$ and $x_{r}$ a value of the solution for some $r>1$. Then there exists a unique $x_1=y_{r}$ where $y_{j+1}= \psi(y_{j})$ and $y_1=x_{r}$.
\end{lema}
\begin{proof}
  We recall that $\varphi(t)= \alpha + \frac{\gamma}{t}$, for $t\neq 0$, has an inverse $\psi(t)= \frac{\gamma}{t - \alpha}$, for $t\neq \alpha$. As $x_{r}$ is a value of the solution for some $r>1$ we see that $x_{j}\neq 0$ for $j=1, ..., r-1$. Therefore, we can reverse the recursion
  $$x_{r}=\varphi(x_{r-1}) \; \Rightarrow \; x_{r-1}=\psi(x_{r})=\psi(y_{1}),$$
  which means that $x_{r-1}=y_{2}$. Proceeding in this same way we obtain $x_{1}=y_{r}$.
\end{proof}

As a consequence we obtain that if $\alpha$ and $\gamma$ are integers or rational numbers then $x_{j}\neq 0$ for all $j$ provided that $x_{1}$ is an  irrational number because the orbit of $y_1=0$ by $\psi$ is a sequence of rational numbers $\{\psi(0), \psi^2(0)=\psi(\psi(0)), \psi^3(0)=\psi(\psi(\psi(0))) \, ...\}$. Although, in applications when we consider a graph with $n$ vertices, we just need to avoid $\{\psi(0), \psi^2(0), \, ...,\psi^n(0) \}$ as initial condition.

\begin{exemplo}\label{example reverse}
   Let $\alpha=-\lambda$ and $\gamma=-1$ be fixed numbers such that $\lambda>2$. We consider the associated recursion
$$b_{j+1}= \varphi(b_{j})= \alpha + \frac{\gamma}{b_{j}}=-\lambda -\frac{1}{b_{j}}$$
and $b_{r+2}=-\frac{2}{\lambda}<0$ for some fixed $r \geq 4$.

In this case, $\Delta= \alpha^2 + 4 \gamma= \lambda^2-4>0$, therefore we have a Type 2 recursion, whose solution is
$$b_{j}=\theta + \frac{\theta'-\theta}{\beta \left(\frac{\theta}{\theta'}\right)^j + 1},$$
where $\theta= \frac{-\lambda}{2} - \frac{1}{2}\sqrt{\lambda^2-4}$ and $\theta'= \frac{-\lambda}{2} + \frac{1}{2}\sqrt{\lambda^2-4}$. It is easy to see that $\theta'+\theta= -\lambda$, $\theta'-\theta=\sqrt{\lambda^2-4}$ and $\theta \theta'=1$ thus
\begin{equation}\label{br2}
b_{j}=\theta + \frac{\theta^{-1} -\theta}{\beta \left(\theta^2\right)^j + 1}.
\end{equation}
As $b_{r+2}=-\frac{2}{\lambda}<0$ we do not have $b_{1}$ available. Even in this case, substituting $b_{r+2}=-\frac{2}{\lambda}$ in \eqref{br2}, by Lemma~\ref{lema reverse}, we can deduce that $\beta=\left(\theta^2\right)^{-r-3}$ therefore
\begin{equation}\label{solution francescoT3T4}
  b_{j}=\theta + \frac{\theta^{-1} -\theta}{\left(\theta^2\right)^{j-r-3} +1}
\end{equation}
is the explicit solution. An easy computation shows that $b_{1}=\theta + \frac{\theta^{-1} -\theta}{\left(\theta^2\right)^{-r-2} +1}$.
\end{exemplo}

\subsection{Extended solutions}
In our applications, the main concern is to predict when $x_{j}$ is positive or negative. Eventually this will tells us the number of eigenvalues below or above the quantity in question. These values compose a discrete set, as $j$ is a natural value. However, as Theorem~\ref{sol general rational difference equation} provides an explicit solution of $ x_{j+1}= \varphi(x_{j}), \; j \geq 1$ that is defined for any $j \in \mathbb{R}$ by $x_{j}= f(j)$ for $j \geq 1$, we can take advantage of that. Except for vertical asymptotes the sign changes occurs when $f(j)=0$. This property has been used in the proof of  \cite[Lemma 4.2]{Jacobs2021}. We give an example that is more illustrative.

\begin{exemplo}\label{ex:extended}
  When $\alpha=1$ and $\gamma=-\frac{1}{4}$ we obtain $\Delta=0$ which is a Type 1  recursion, whose solution is
$$ x_{j}=f(j)=\theta\, \left(1+ \frac{1}{(\beta +j )}\right),$$
where $\theta= \frac{\alpha}{2}=\frac{1}{2}$ and $\beta \in \mathbb{R}$ is defined by the initial point $x_1$. Choosing $x_1=\frac{1}{2}\left(1+\frac{1}{-5+\sqrt{2}}\right)\simeq 0.36$ in such way that $\beta=-6+\sqrt{2}$, we can have change of signal when $\beta +j=0$ producing $j=-\beta=6-\sqrt{2} \approx  4.58$, a vertical asymptote of $f$, or when $1+ \frac{1}{(\beta +j )}=0$ producing $j=-\beta-1=5-\sqrt{2} \approx  3.58$, a zero of the function $f$. Neither of this values is an integer thus $x_{j} \neq 0$ for all $j$.

From the general formula $ x_{j}=\frac{1}{2}\, \left(1+ \frac{1}{(-6+\sqrt{2} +j )}\right)$. We know that $\displaystyle \lim_{j\to \infty} x_{j} =\frac{1}{2}>0$. By continuity of $f$ we know that $ x_{j}>0$ for $j \geq 5$. $\displaystyle \lim_{j\to -\infty} x_{j} =\frac{1}{2}>0$. By continuity of $f$ we know that $ x_{j}>0$ for $j \leq 3$.  Thus, the only possible negative therm is $ x_{4}\approx -0.35$ because $j=4$ is the only integer between $3.58$ and $4.58$.
\end{exemplo}

\subsection{Attracting and repelling sets}
Given the sequence $x_{j+1}= \varphi(x_{j}), \; j \geq 1$  and a fixed point $\theta$, that is, $\varphi(\theta)= \theta$, the iteration of an initial condition contained in some interval $I$ will be repelled from $\theta$, resp. attracted to $\theta$, according to whether $|\varphi'|>1$ resp. $|\varphi'|<1$, in $I$.
$$|x_{j+1}-\theta|=|\varphi(x_{j})-\varphi(\theta)|=|\varphi'(\xi)|\,|x_{j}-\theta|,$$
for some $\xi$ between $x_{j}$ and $\theta$. Therefore
\begin{equation}\label{contraction}
  \begin{cases}
    |x_{j+1}-\theta|<|x_{j}-\theta|, &\text{ if }  |\varphi'(\xi)|<1;\\
    |x_{j+1}-\theta|>|x_{j}-\theta|, & \text{ if }   |\varphi'(\xi)|>1.
\end{cases}
\end{equation}

As $\varphi'(t)= -\frac{\gamma}{t^2}$ and $\gamma \neq 0$ we have two points $\xi=\pm \sqrt{|\gamma|}$ where $|\varphi'(\xi)|=1$. On the other points $|\varphi'(\xi)|<1$ ($t>\sqrt{|\gamma|}$ or $t<-\sqrt{|\gamma|}$) or $|\varphi'(\xi)|>1$ ($-\sqrt{|\gamma|}<t < \sqrt{|\gamma|}$)).

The iterates  $x_{j+1}= \varphi(x_{j}), \; j \geq 1$ not only are attracted to or are repelled from  the fixed points but they do it in a monotone way \ero{when $\gamma<0$} because the sign of $\varphi'(t)= -\frac{\gamma}{t^2}$ depends only on the sign of $\gamma$.  \ero{When  $\gamma>0$ the map is ordering reversing, thus the convergence to the fixed point will be alternating.}

\begin{exemplo}\cite{Belardo2019}\label{exemplo Francesco}
  Let $\alpha=-\lambda$ and $\gamma=-1$ be fixed numbers such that $\lambda>2$. We consider the associated recursion
$$z_{j+1}= \varphi(z_{j})= \alpha + \frac{\gamma}{z_{j}}=-\lambda -\frac{1}{z_{j}}$$
and $z_{1}=-\lambda<0$.

In this case, $\Delta= \alpha^2 + 4 \gamma= \lambda^2-4>0$, therefore we have a Type 2 equation whose solution is
$$z_{j}=\theta + \frac{\theta'-\theta}{\beta \left(\frac{\theta}{\theta'}\right)^j + 1},$$
where $\theta= \frac{-\lambda}{2} - \frac{1}{2}\sqrt{\lambda^2-4}$ and $\theta'= \frac{-\lambda}{2} + \frac{1}{2}\sqrt{\lambda^2-4}$. It is easy to see that $\theta'+\theta= -\lambda$, $\theta'-\theta=\sqrt{\lambda^2-4}$ and $\theta \theta'=1$ thus
$$z_{j}=\theta + \frac{\theta^{-1} -\theta}{\beta \left(\theta^2\right)^j + 1}.$$
As $z_{1}=-\lambda$ we can deduce that $\beta=-1$ therefore
\begin{equation}\label{solution francesco}
  z_{j}=\theta - \frac{\theta^{-1} -\theta}{\left(\theta^2\right)^j -1}
\end{equation}
is the explicit solution.
\begin{figure}[H]
  \centering
  \includegraphics[width=10cm]{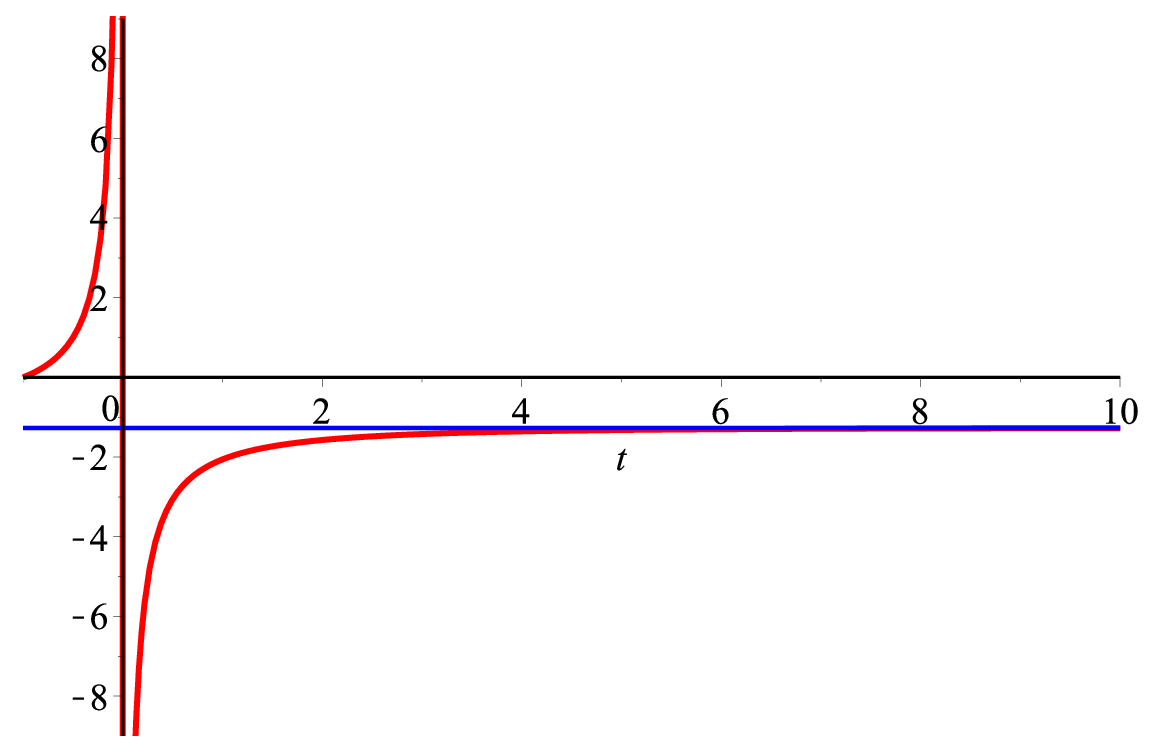}
  \caption{ Function $f(j)=\theta - \frac{\theta^{-1} -\theta}{\left(\theta^2\right)^j -1}$ for $\lambda=\sqrt{2+\sqrt{5}}>2$. In this case, $\theta\approx -1.27202$,  $\theta'\approx -0.786145$, $\beta=-1$. $z_{j}=f(j) \to \theta$ increasingly when $j\to \infty$ }\label{exem_francesco_index}
\end{figure}
A quick examination shows that $\theta^{-1} -\theta>0$ and  $\theta^2> 1$ because $\lambda>2$. Thus $z_{j}<\theta$ and $z_{j}\to \theta$ increasingly when $j\to \infty$ if $z_{1}<\theta$ which is the case for $z_{1}=-\lambda$.

\end{exemplo}

\begin{obs}
We can summarize all the other possibilities for different values of $z_{1}$, when $\alpha=-\lambda$ and $\gamma=-1$ be fixed numbers such that $\lambda>2$, as follows. If $\theta<z_{1}<\theta'$ then $z_{j}\to \theta$ decreasingly when $j\to \infty$ by \eqref{contraction}. If $\theta'<z_{1}<0$ then $z_{j}$ is increasing by \eqref{contraction} and, in some point, $z_{j}>0$. Finally, if $z_{1}>0$ then $z_{2}=-\lambda -\frac{1}{z_{1}} < -\lambda$ therefore $z_{j}\to \theta$ increasingly when $j\to \infty$.
\end{obs}

\color{blue}
\subsection{Classifying the dynamics according to the stability of fixed points}
In this section we will carefully analyze the possibilities of behavior for a solution of the rational recurrence $x_{j+1}=\alpha +\frac{\gamma}{x_{j}}$ for $0\neq x_1 \not\in \mathcal{Z}$, where $\varphi(t)=\alpha +\frac{\gamma}{t}, \; t \neq 0$, depending on the combination of values of $\alpha$ and $\gamma$ summarizing our previous analyzes.  We recall that the fixed points of $\varphi$ are invariant sets for the dynamics containing the information on the asymptotic behavior of the solutions.  A fixed point $\theta$ ($\varphi(\theta)= \theta$) is hyperbolic attracting (resp. repelling) for $\varphi$ if  $|\varphi'(\theta)|<1$(resp. $|\varphi'(\theta)|>1$) in some neighborhood of $(\theta-\varepsilon, \theta+\varepsilon)$. When $|\varphi'(\theta)|=1$ we have a saddle. In the one dimensional case, a saddle could be laterally repelling  or attracting depending on whether $|\varphi'|$ is smaller/bigger than one on the chosen neighborhood $(\theta-\varepsilon, \theta)$ (left) or $(\theta, \theta+\varepsilon)$(right).

We recall that,  for $\Delta=0$ (Type 1) we have one  fixed point, for $\Delta>0$ (Type 2) we have two  fixed points $\theta$ and $\theta'$. However, for $\Delta=\alpha^2+ 4\gamma <0$ (Type 3) we have no fixed points. We will analyze all Type 1, which has only two possibilities (the case  $\alpha=0$ produces $\gamma=0$,  because $\Delta=\alpha^2+ 4\gamma =0$), and Type 2 cases (excluding the trivial solutions, $\gamma=0$, because we must have a first order equation) which has five possibilities (the case  $\alpha=0$ and $\gamma<0$ produces $\Delta=0^2+ 4\gamma <0$ having no fixed points, so it is of Type 3).

\subsubsection{Type 1  dynamics}
The Type 1 cases are the following:
\begin{enumerate}
	\item  $\alpha<0$, $\gamma=-\frac{\alpha^2}{4}<0$;
	\item  $\alpha>0$, $\gamma=-\frac{\alpha^2}{4}<0$.
\end{enumerate}
In order to avoid redoing all the process of analyzing the dynamical behavior of orbits, we notice that it is possible to explore its  symmetry relying on the previous cases.    We recall from Lemma~\ref{lem:symmetry}  that   $x_1=a$ and  $x_{j+1}=\varphi(x_j) $, where   $\varphi(t)=\alpha +\frac{\gamma}{t}, \; t\neq 0$, if and only if,  $y_1=-a$ and  $y_{j+1}=\psi(y_j) $, where   $\psi(t)=-\alpha +\frac{\gamma}{t}, \; t\neq 0$. Thus we need to analyze just the Case 1.

\textbf{Case 1:}  In this case we have one fixed point, because  $\alpha<0$, $\gamma=-\frac{\alpha^2}{4}<0$ produces $\Delta=0$.

 Here we have $\varphi(t)=\alpha+ \frac{\gamma}{t}, \; t \neq 0$ whose graph is given by Figure~\ref{fig:exp_type1}
 \begin{figure}[H]
 	\centering
 	\includegraphics[width=6cm]{ 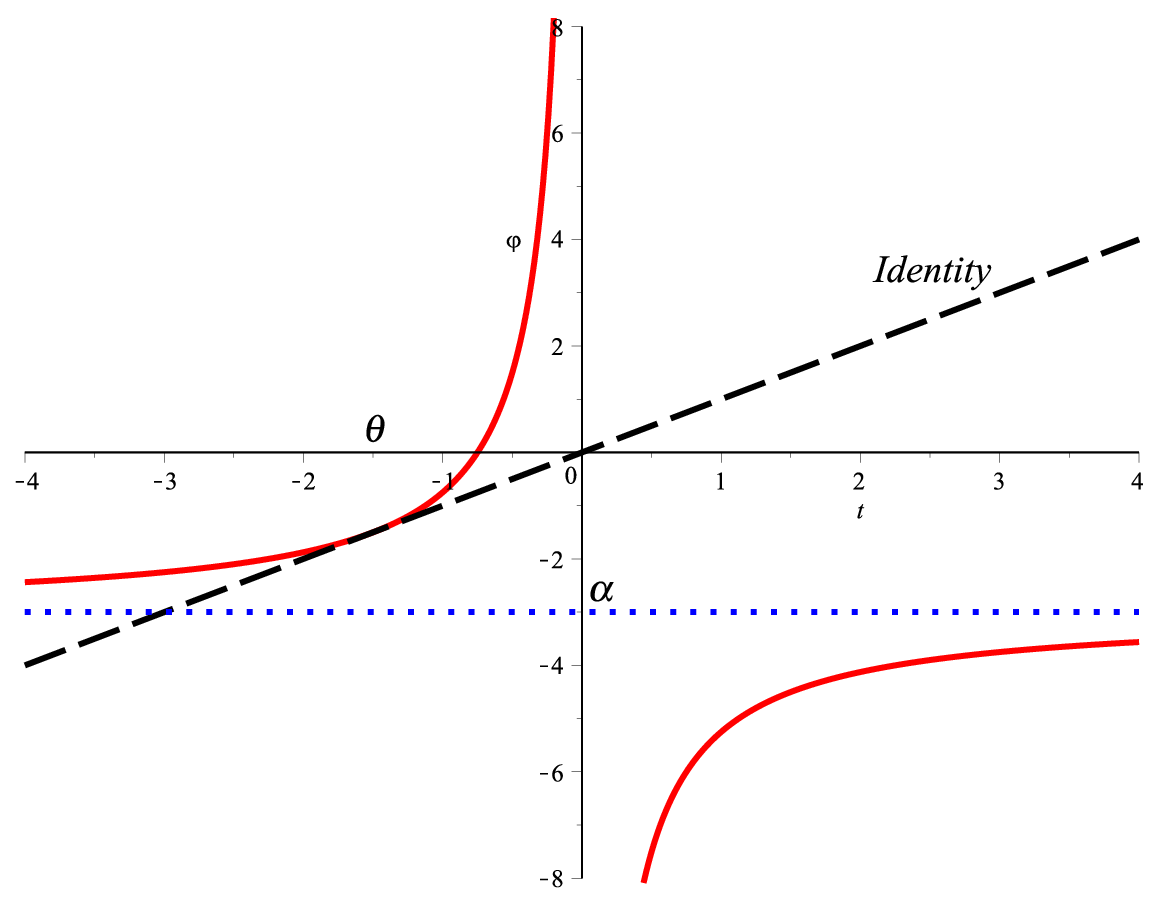} \quad \includegraphics[width=6cm]{ 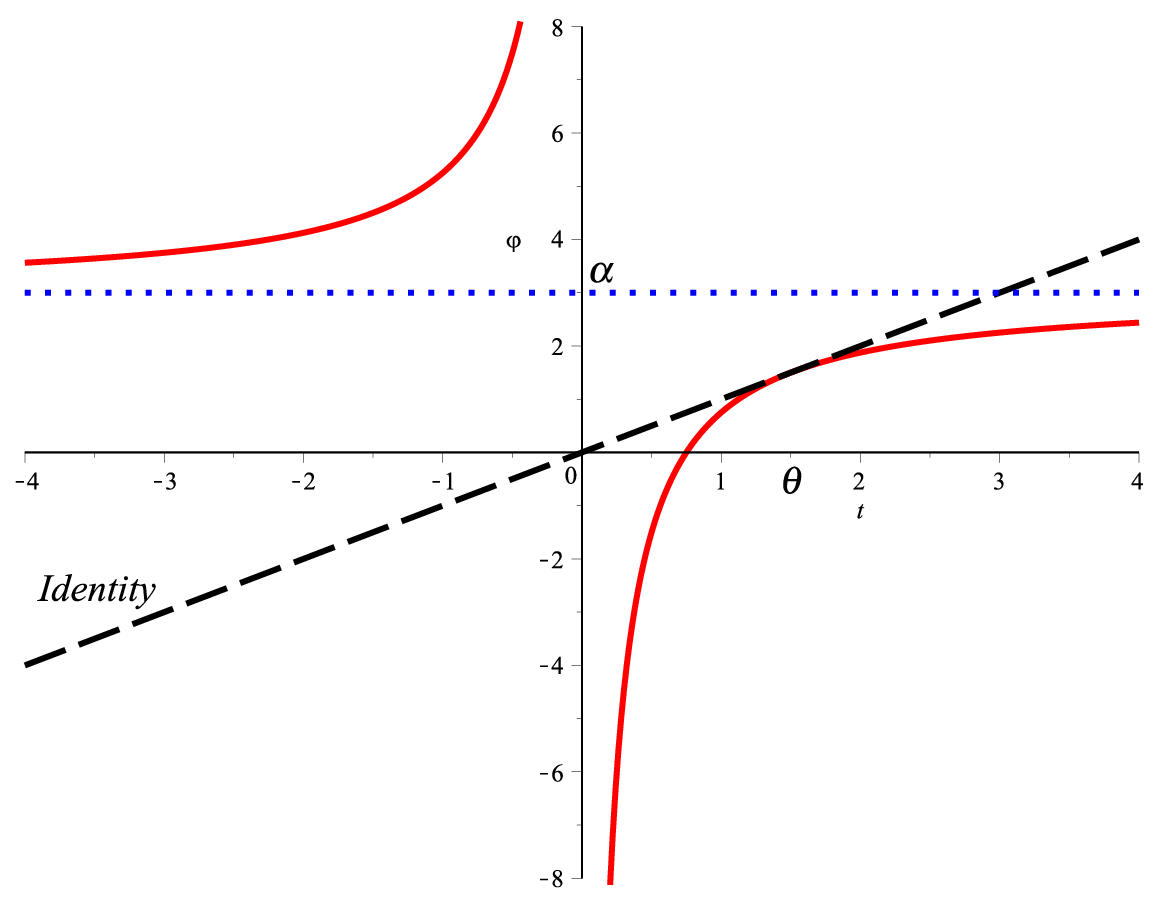}
 	\caption{ Dynamics of $\varphi(t)=-3 -\frac{9}{4 t}, \; t \neq 0$ (left)  and $\varphi(t)=3 -\frac{9}{4 t}, \; t \neq 0$ (right).}\label{fig:exp_type1}
 \end{figure}
 Thus $\theta= \frac{\alpha}{2} <0$ is the unique fixed point.

 We notice that $|\varphi'(t)| =|-\frac{\gamma}{t^2}| <1$  for $t<-\sqrt{-\gamma}=\frac{\alpha}{2}$  or $t>-\frac{\alpha}{2}$.  As a consequence, $\theta$ is a saddle laterally attracting/repelling fixed point ($|\varphi'(\theta)| =1$).  More precisely, define $I=(-\infty, \theta)$ and $I'=(\theta,0)$.

 If $x_1 \in I$ then $x_{j+1}= \varphi(x_j)$  will increasingly converge to $\theta$, because the map is monotonous and contracts distances.

 Notice that, if $t>0$ then $\varphi(t) < \alpha < \theta$ meaning that, these points will be attracted to $\theta$, that is $x_{j+1}= \varphi(x_j)$, monotonously increasing for $j \geq 2$.

Additionally, for $t \in I'=(\theta, 0)/\mathcal{Z}$ as $|\varphi'(t)|>1$ we obtain an increasing sequence of iterates that necessarily become positive for some power  $m>1$. Otherwise, if it was bounded by 0,  it would produce a limit point $c=\lim_{j\to \infty} \varphi^j(t)\leq 0$. If $c=0$ we get a contradiction with the fact that $|\varphi^j(t) -\theta|$ must be unbounded due to the expansiveness of $\varphi$. Thus $c<0$ which, by continuity of $\varphi$, would produce  a second fixed point($\varphi(c)=c > \theta$), an absurd. Consequently, there will be a power $m$ such that $c=\varphi^m(t)>  0$.
 Since for any $\varphi^m(t)>0$, we obtain $\varphi(\varphi^m(t))=\alpha+\frac{\gamma}{\varphi^m(t)}<\alpha< \theta$, or equivalently,  $\varphi^{m+1}(t) \in I $, all the future iterates $m+1, m+2,\ldots$  will be increasingly convergent to $\theta$.

 In order to determine the null set $\mathcal{Z}$, we consider the point $\varphi^{-1}(0)=-\frac{\gamma}{\alpha} =\frac{\alpha}{4} <0$. Since $\varphi^{-1}(0) \in I'$, all the reversal iterates $\varphi^{-2}(0),\varphi^{-3}(0),\ldots$ will monotonously decrease and converge to $\theta$ (because the inverse function is contracting in this interval and $\varphi$ is increasing). In this way we obtain that $\mathcal{Z}=(b_j)_{j\in \mathbb{N}}$ is a sequence of points  contained in $(\theta, \frac{\alpha}{4} ]$ with $b_1=\frac{\alpha}{4} $ and $\lim_{j\to \infty} b_j =\theta$.

We summarize this case as follows
\begin{prop} \label{prop: dyn Type 1 Delta zero} Let  $\alpha<0$ and  $\gamma=-\frac{\alpha^2}{4}<0$ be fixed numbers such that  $\Delta=\alpha^2+ 4\gamma =0$. Then,
	\begin{enumerate}
		\item  $\theta=  \frac{\alpha}{2}$ is an attracting fixed point for  $t\in  (-\infty, \theta)$ and  a repelling fixed point for  $t \in  (\theta, 0)$ .
		\item  $\mathcal{Z}=(b_j)_{j\in \mathbb{N}}$ is a sequence of points  contained in $(\theta, \frac{\alpha}{4}]$ with $b_1=\frac{\alpha}{4}$ and $\lim_{j\to \infty} b_j =\theta$, decreasingly.
		\item If $x_1 \in (-\infty, \theta)$ then $x_j$ increases and converges to $\theta$ for $j \geq 1$.
		\item If $x_1 >0$ then $x_j$  increases and converges to  $\theta$ for $j \geq 2$.
		\item If $x_1 \in (\theta, 0)$ then, there exists $m \in \mathbb{N}$ such that, $x_j$ increases for $j=1,\ldots, m-1$ and $x_m>0$. Then, $x_j$ increases and converges to $\theta$ for $j \geq m+1$.
		\item If $x_1=\theta$ then $x_j=\theta$  for $j \geq 1$.
	\end{enumerate}
\end{prop}

\subsubsection{Type 2 dynamics}
The Type 2 cases are the following:
\begin{enumerate}
	\item  $\alpha=0$, $\gamma>0$;
	\item  $\alpha<0$, $\gamma<0$;
	\item  $\alpha<0$, $\gamma>0$;
	\item  $\alpha>0$, $\gamma<0$;
	\item  $\alpha>0$, $\gamma>0$.
\end{enumerate}

\textbf{Case 1:}  This case $\alpha=0$, $\gamma>0$ produces two fixed points because  $\Delta=0^2+ 4\gamma >0$.\\
Here we have $\varphi(t)=\frac{\gamma}{t}, \; t \neq 0$ whose graph is given by Figure~\ref{fig:exp_zero_pos}
\begin{figure}[H]
	\centering
	\includegraphics[width=8cm]{ 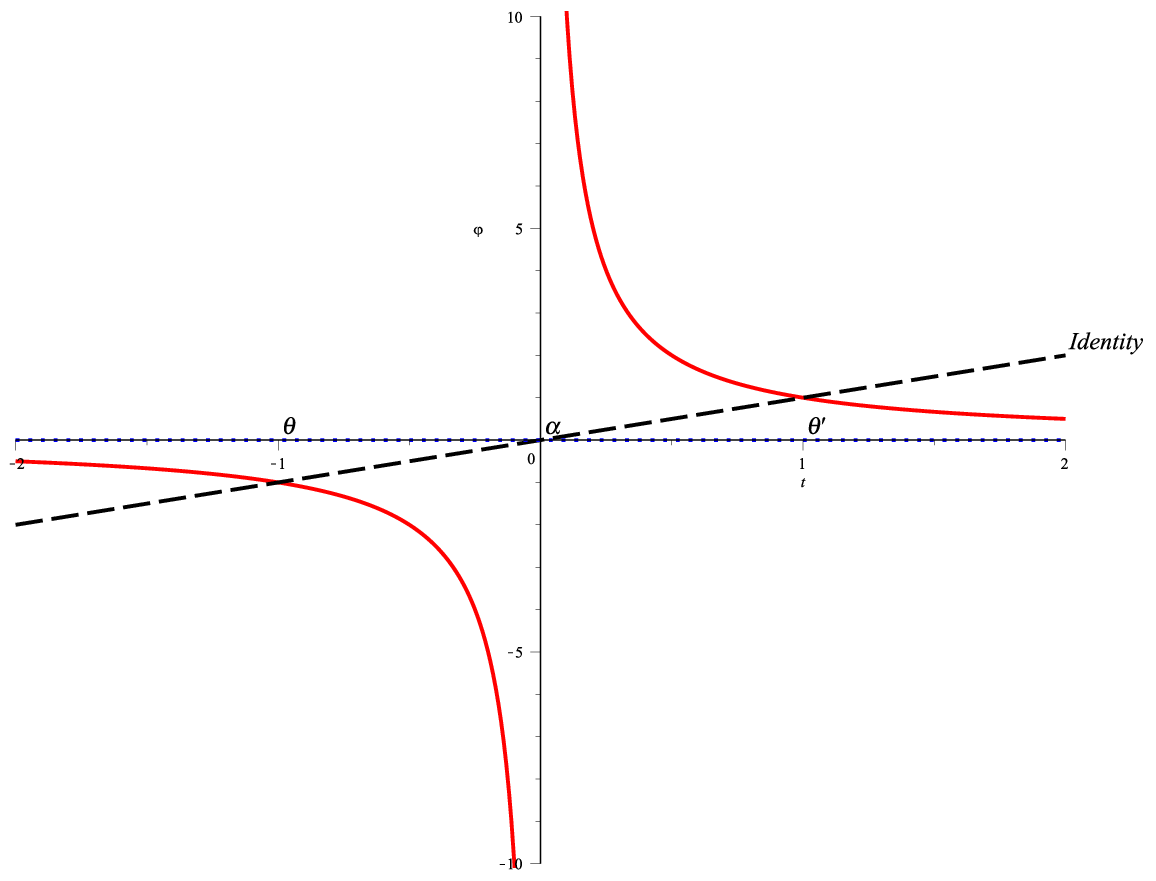}
	\caption{ Dynamics of $\varphi(t)=0 +\frac{1}{t}, \; t \neq 0$.}\label{fig:exp_zero_pos}
\end{figure}
Thus $\theta= \frac{0}{2} - \frac{1}{2}\sqrt{0^2+ 4\gamma} = -\sqrt{\gamma}$ and $\theta'= \frac{0}{2} +\frac{1}{2}\sqrt{0^2+ 4\gamma} = \sqrt{\gamma}$, so that $\theta < 0<\theta'$. The null set is just $\mathcal{Z}=\emptyset$ because $\varphi(t)=0$ has no solution.\\
We notice that $|\varphi'(t)| =|-\frac{\gamma}{t^2}| <1$  for $t<-\sqrt{\gamma}$  or $t>\sqrt{\gamma}$.  As a consequence, $\theta , \theta'$ are saddles, that is, laterally attracting/repelling fixed points($|\varphi'(\theta)|=|\varphi'(\theta')| =1$).   Finally, if $x_1 \in (0, \sqrt{\gamma}) $ then $x_2= \frac{1}{x_1}>\sqrt{\gamma}$ (because  $\varphi$ decreasing for $0<t< \sqrt{\gamma}$, so $x_1< \sqrt{\gamma}$ means that $x_2> \varphi(\sqrt{\gamma})= \sqrt{\gamma}$) and $x_j, j \geq 2$ decreases to $\theta'$. Otherwise, if $x_1 \in (-\sqrt{\gamma},0) $ then $x_2= \frac{1}{x_1}<-\sqrt{\gamma}$ and $x_j, j \geq 2$ increases to $\theta$. \\
We summarize this case as follows
\begin{prop} \label{prop: dyn zero neg Delta posit} Let  $\alpha=0$ and  $\gamma<0$ be fixed numbers such that  $\Delta=\alpha^2+ 4\gamma >0$. Then,
	\begin{enumerate}
		\item  $\theta= -\sqrt{\gamma}$ is an attracting fixed point for  $t<-\sqrt{\gamma}$.
		\item  $\theta'= \sqrt{\gamma}$ is an attracting fixed point for  $t>\sqrt{\gamma}$ .
		\item  $\mathcal{Z}=\emptyset$.
		\item If $x_1 \in (\sqrt{\gamma}, +\infty)$ then $x_j$ decreases and converges to $\theta'$ for $j \geq 1$(resp.
		if $x_1 \in (0, \sqrt{\gamma})$ then $x_j$ decreases and converges to $\theta'$ for $j \geq 2$).
		\item If $x_1 \in (-\infty, -\sqrt{\gamma})$ then $x_j$ increases and converges to $\theta$ for $j \geq 1$(resp.
		if $x_1 \in (-\sqrt{\gamma},0)$ then $x_j$ increases and converges to $\theta$ for $j \geq 2$.)
		\item  If $x_1=\theta$ (resp. $x_1=\theta$) then $x_j=\theta$(resp. $x_j=\theta'$) for $j \geq 1$.
	\end{enumerate}
\end{prop}

\textbf{Case 2:}  This case  $\alpha<0$, $\gamma<0$ produces two fixed points, because $\Delta=\alpha^2+ 4\gamma >0$.\\
Here we have $\varphi(t)=\alpha+\frac{\gamma}{t}, \; t \neq 0$ whose graph is given by Figure~\ref{fig:exp_neg_neg}
\begin{figure}[H]
	\centering
	\includegraphics[width=8cm]{ 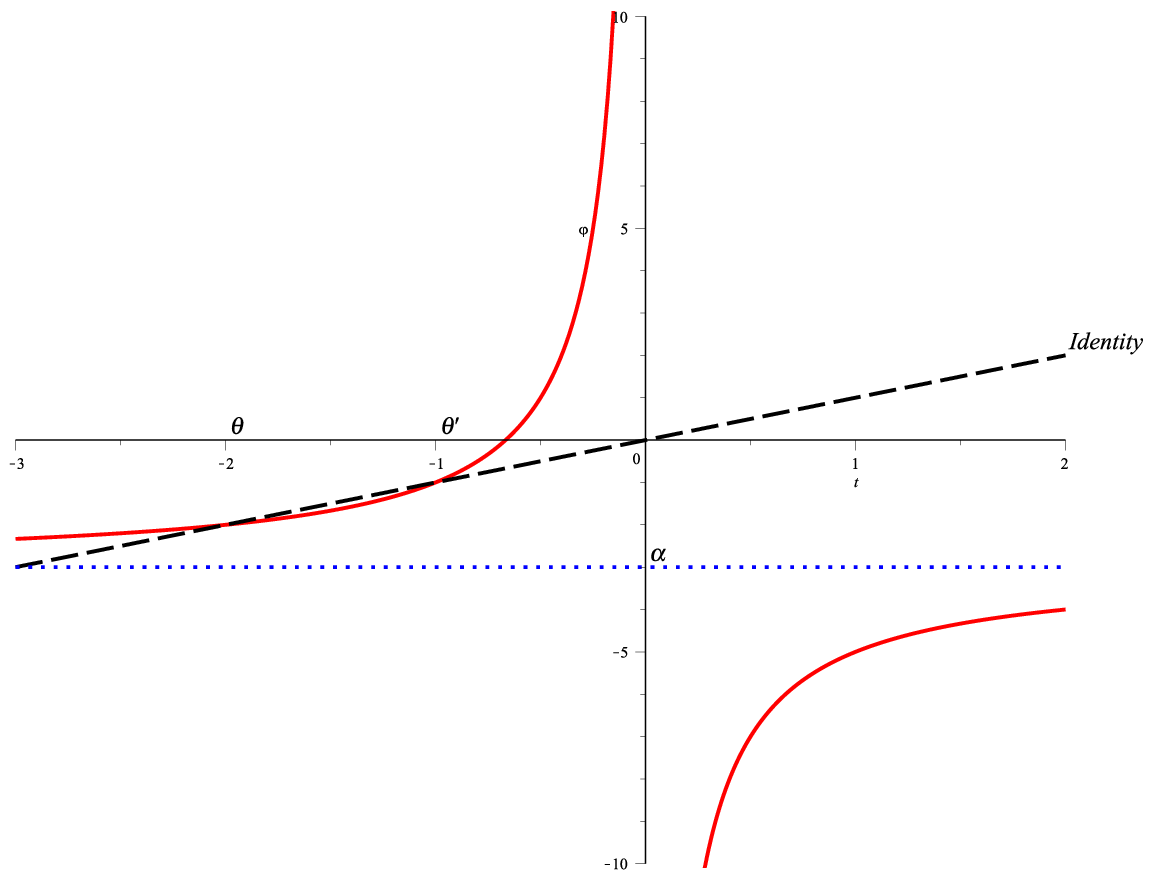}
	\caption{ Dynamics of $\varphi(t)=-3 -\frac{2}{t}, \; t \neq 0$.}\label{fig:exp_neg_neg}
\end{figure}
Thus $\theta= \frac{\alpha}{2} - \frac{1}{2}\sqrt{\alpha^2+ 4\gamma} $ and $\theta'= \frac{\alpha}{2} +\frac{1}{2}\sqrt{\alpha^2+ 4\gamma} $, so that $\theta < \theta'<0$, because $\theta \theta'=-\gamma>0$.

We notice that $\varphi'(t)=-\frac{\gamma}{t^2}>0,\, t\neq 0$, so it is monotonously increasing in each domain of continuity and  $|\varphi'(t)| =|-\frac{\gamma}{t^2}| <1$  for $t<-\sqrt{-\gamma}$  or $t>\sqrt{-\gamma}$ (because $|\gamma|=-\gamma$). Consequently, $|\varphi'(t)|>1$   for $t>-\sqrt{-\gamma}$. Define $I=(-\infty, -\sqrt{-\gamma})$ and $I'=(-\sqrt{-\gamma}, 0)$. We claim that $\theta \in I$ and $\theta' \in I'$.  Since both are negative, we just need to check that $|\varphi'(\theta)|<1$ and  $|\varphi'(\theta')|>1$. Indeed,
\[\varphi'(\theta) = -\frac{\gamma}{\theta^2} = \frac{\theta\theta'}{\theta^2} =\frac{\theta'}{\theta}= \frac{\alpha -\theta}{\theta}<1 \]
if, an only if $\alpha -\theta > \theta$, or equivalently $\theta < \frac{\alpha}{2}$ which is always true. Analogously,
\[\varphi'(\theta') = -\frac{\gamma}{(\theta')^2} = \frac{\theta\theta'}{(\theta')^2} =\frac{\theta}{\theta'}= \frac{\alpha -\theta'}{\theta'}>1 \]
if, and only if, $\alpha -\theta' < \theta'$, or equivalently $\theta' > \frac{\alpha}{2}$ which is always true.
As a consequence, $\theta , \theta'$ are hyperbolic attracting/repelling fixed points with respect to the intervals $I$ and $I'$.   Notice that, if $t=-\sqrt{-\gamma}$ as $\varphi(t) < t$ for $t \in (\theta, \theta')$ we get $\varphi(-\sqrt{-\gamma}) <-\sqrt{-\gamma} \in I$. Meaning that, the points repelled from $I'$ will cross to $I$  and then be attracted to $\theta$. Additionally, for $t \in (\theta', 0)/\mathcal{Z}$ we obtain an increasing sequence of iterates that necessarily become positive for some power  $m>1$. Otherwise, if it was bounded by 0,  it would produce a limit point $c=\lim_{j\to \infty} \varphi^j(t)\leq 0$. If $c=0$ we get a contradiction with the fact that $|\varphi^j(t) -\theta'|$ must be unbounded due to the expansiveness of $\varphi$. Thus $c<0$ which, by continuity of $\varphi$, would produce  a third fixed point($\varphi(c)=c$), an absurd. Consequently, there will be a power $m$ such that $c=\varphi^m(t)>  0$.
Since for any $\varphi^m(t)>0$, we obtain $\varphi(\varphi^m(t))=\alpha+\frac{\gamma}{\varphi^m(t)}<\alpha< \theta$, or equivalently,  $\varphi^{m+1}(t) \in I $, all the future iterates $m+1, m+,\ldots$  will be increasingly convergent to $\theta$.

In order to determine the null set $\mathcal{Z}$, we consider the point $\varphi^{-1}(0)=-\frac{\gamma}{\alpha} <0$. We claim that $-\frac{\gamma}{\alpha}> \theta'$. Indeed,
\[-\frac{\gamma}{\alpha}> \theta' \iff \frac{\theta\, \theta'}{\alpha}>\theta' \iff \theta > \alpha=\theta+\theta'\]
which is always true. Thus, all the reversal iterates $\varphi^{-2}(0),\varphi^{-3}(0),\ldots$ will monotonously decrease and converge to $\theta'$. In this way we obtain that $\mathcal{Z}=(b_j)_{j\in \mathbb{N}}$ is a sequence of points  contained in $(\theta', -\frac{\gamma}{\alpha}]$ with $b_1=-\frac{\gamma}{\alpha}$ and $\lim_{j\to \infty} b_j =\theta'$.

We summarize this case as follows
\begin{prop} \label{prop: dyn neg neg Delta posit} Let  $\alpha<0$ and  $\gamma<0$ be fixed numbers such that  $\Delta=\alpha^2+ 4\gamma >0$. Then,
	\begin{enumerate}
		\item  $\theta=  \frac{\alpha}{2} - \frac{1}{2}\sqrt{\alpha^2+ 4\gamma}$ is an attracting fixed point for  $t\in I$.
		\item  $\theta'= \frac{\alpha}{2} + \frac{1}{2}\sqrt{\alpha^2+ 4\gamma}$ is a repelling fixed point for  $t \in I'$ .
		\item  $\mathcal{Z}=(b_j)_{j\in \mathbb{N}}$ is a sequence of points  contained in $(\theta', -\frac{\gamma}{\alpha}]$ with $b_1=-\frac{\gamma}{\alpha}$ and $\lim_{j\to \infty} b_j =\theta'$, decreasingly.
		\item If $x_1 \in (-\infty, \theta)$ then $x_j$ increases and converges to $\theta$ for $j \geq 1$(resp.
		if $x_1 \in (\theta,\theta')$ then $x_j$ decreases and converges to $\theta$ for $j \geq 1$.)
		\item If $x_1 \in (\theta', 0)$ then, there exists $m \in \mathbb{N}$ such that, $x_j$ increases for $j=1,\ldots, m-1$ and $x_m>0$. Then, $x_j$ increases and converges to $\theta$ for $j \geq m+1$.
		\item If $x_1=\theta$ (resp. $x_1=\theta$) then $x_j=\theta$(resp. $x_j=\theta'$) for $j \geq 1$.
	\end{enumerate}
\end{prop}

\textbf{Case 3:}  This case $\alpha<0$, $\gamma>0$ produces two fixed points because $\Delta=\alpha^2+ 4\gamma >0$.\\
Here we have $\varphi(t)=\alpha+\frac{\gamma}{t}, \; t \neq 0$ whose graph is given by Figure~\ref{fig:exp_neg_pos}
\begin{figure}[H]
	\centering
	\includegraphics[width=8cm]{ 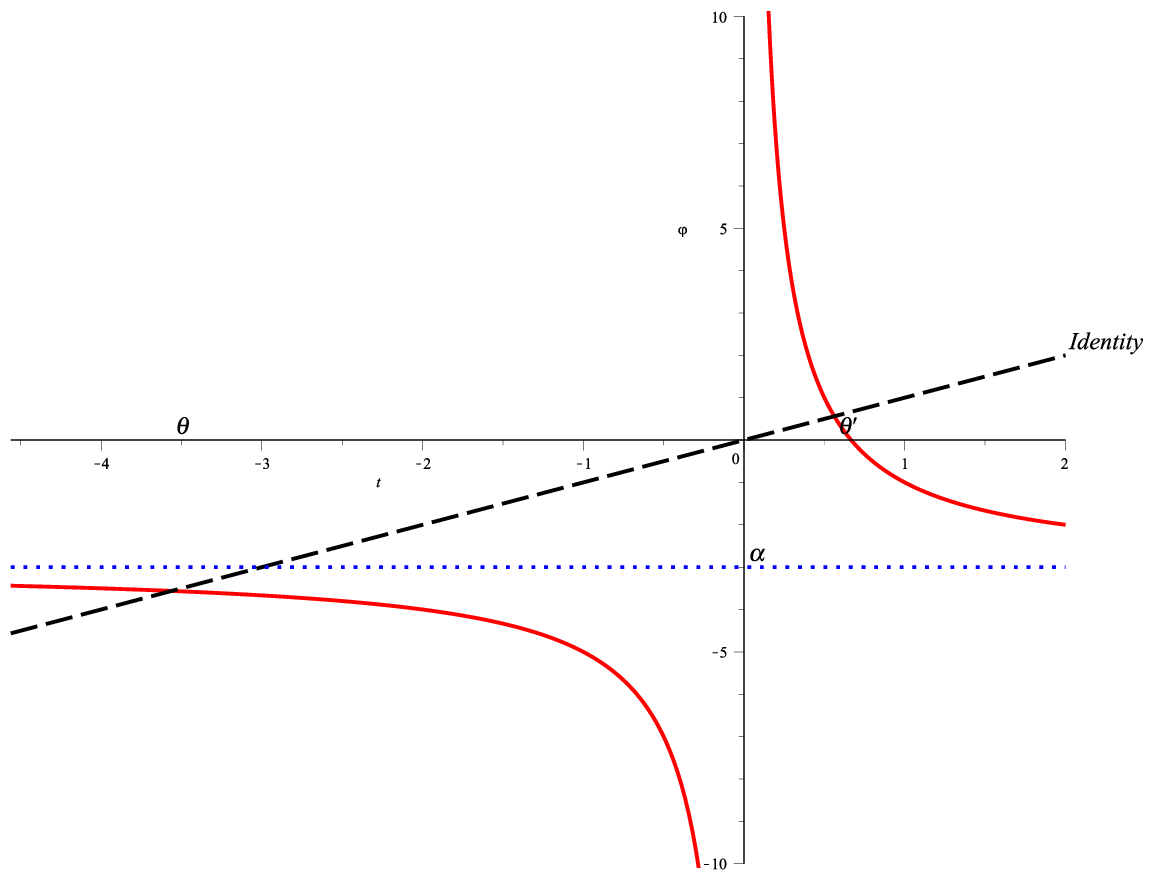}
	\caption{ Dynamics of $\varphi(t)=-3 +\frac{2}{t}, \; t \neq 0$.}\label{fig:exp_neg_pos}
\end{figure}
Thus $\theta= \frac{\alpha}{2} - \frac{1}{2}\sqrt{\alpha^2+ 4\gamma} $ and $\theta'= \frac{\alpha}{2} +\frac{1}{2}\sqrt{\alpha^2+ 4\gamma} $, so that $\theta < 0 < \theta'$, because $\theta \theta'=-\gamma<0$.

We notice that $\varphi'(t)=-\frac{\gamma}{t^2}<0,\, t\neq 0$, so it is monotonously decreasing in each domain of continuity and  $|\varphi'(t)| =|-\frac{\gamma}{t^2}| <1$  for $t<-\sqrt{\gamma}$  or $t>\sqrt{\gamma}$. Consequently, $|\varphi'(t)|>1$   for $-\sqrt{\gamma} <t <\sqrt{\gamma}$ assuming $t\neq 0$.

Define $I=(-\infty, -\sqrt{\gamma}) $and $I'=(0, \sqrt{\gamma}) $.

We claim that $\theta \in I$ and $\theta' \in I'$.  Considering the signal of each one, we just need to check that $\varphi'(\theta)>-1$ and  $\varphi'(\theta')<-1$. Indeed,
\[\varphi'(\theta) = -\frac{\gamma}{\theta^2} = \frac{\theta\theta'}{\theta^2} =\frac{\theta'}{\theta}= \frac{\alpha -\theta}{\theta}>-1 \]
if, an only if $\alpha -\theta <- \theta$, or equivalently $\alpha< 0$ which is always true. Analogously,
\[\varphi'(\theta') = -\frac{\gamma}{(\theta')^2} = \frac{\theta\theta'}{(\theta')^2} =\frac{\theta}{\theta'}= \frac{\alpha -\theta'}{\theta'}<-1 \]
if, and only if, $\alpha -\theta' < -\theta'$, or equivalently $\alpha <0$ which is always true. As a consequence, $\theta , \theta'$ are hyperbolic attracting/repelling fixed points with respect to the intervals $I$ and $I'$.

We notice that for $\theta < t <0$ we have $\varphi(t)< \varphi(\theta)=\theta \in I$, so that all the future iterates are (not monotonously) attracted to $\theta$.

In order to determine the null set $\mathcal{Z}$, we consider the point $c=\varphi^{-1}(0)=-\frac{\gamma}{\alpha} <0$.
We claim that $c \in I'$. Indeed,
\[\varphi'(c) = -\frac{\gamma}{(-\frac{\gamma}{\alpha})^2} = \frac{-\alpha^2}{\gamma} <\frac{-4\gamma}{\gamma}< -1,\]
because $\Delta=\alpha^2+ 4\gamma >0 \iff -\alpha^2 < -4 \gamma$.
As $\theta'$ is a repelling point, , we obtain  that all the future reverse iterates of $\varphi^{-k}(0)$ are (not monotonously) attracted to $\theta'$. In this way we obtain that $\mathcal{Z}=(b_j)_{j\in \mathbb{N}}$ is a sequence of points  contained in $(0, -\frac{\gamma}{\alpha}]$ with $b_1=-\frac{\gamma}{\alpha}$ and $\lim_{j\to \infty} b_j =\theta'$.

If $t > -\frac{\gamma}{\alpha}$ we obtain $\varphi(t)<0$ so the future iterates will be convergent to $\theta$ by the previous analyzed cases.

Finally, we consider $t \in (0, -\frac{\gamma}{\alpha}) /\mathcal{Z}$ and notice that $(0, -\frac{\gamma}{\alpha}) \subset  I'$ thus there will be a power $m$ such that $\varphi^m(t) \not\in I'$(because $|\varphi^j(t) -\theta'|$ must be unbounded due to the expansiveness of $\varphi$).

Does not matter if  $\varphi^m(t)<0$ or $\varphi^m(t)> -\frac{\gamma}{\alpha}$,  all the future iterates $m+1, m+2,\ldots$  will be convergent to $\theta$ by the previous analyzed cases.

We summarize this case as follows
\begin{prop} \label{prop: dyn neg pos Delta posit} Let  $\alpha<0$ and  $\gamma>0$ be fixed numbers such that  $\Delta=\alpha^2+ 4\gamma >0$. Then,
	\begin{enumerate}
		\item  $\theta=  \frac{\alpha}{2} - \frac{1}{2}\sqrt{\alpha^2+ 4\gamma}<0$ is an attracting fixed point for  $t\in I_-$.
		\item  $\theta'= \frac{\alpha}{2} + \frac{1}{2}\sqrt{\alpha^2+ 4\gamma}$ is a repelling fixed point for  $t \in I'_+$ .
		\item  $\mathcal{Z}=(b_j)_{j\in \mathbb{N}}$ is a sequence of points  contained in $(0, -\frac{\gamma}{\alpha}]$ with $b_1=-\frac{\gamma}{\alpha}$ and $\lim_{j\to \infty} b_j =\theta'$.
		\item If $\theta\neq x_1 \in (-\infty, 0) \cup (-\frac{\gamma}{\alpha}, +\infty)$ then $x_j$  converges(not monotonously) to $\theta$ for $j \geq 1$
		\item If $\theta'\neq x_1 \in (0,-\frac{\gamma}{\alpha})/\mathcal{Z}$ then $x_j$  converges(not monotonously) to $\theta$.
		\item If $x_1=\theta$ (resp. $x_1=\theta$) then $x_j=\theta$(resp. $x_j=\theta'$) for $j \geq 1$.
	\end{enumerate}

\end{prop}

\textbf{Cases 4 and 5:} This cases  $\alpha<0$ (resp. $\alpha>0$), $\gamma>0$ produces two fixed points because $\Delta=\alpha^2+ 4\gamma >0$.\\
Here we have $\varphi(t)=\alpha+\frac{\gamma}{t}, \; t \neq 0$ whose graph is given by Figure~\ref{fig:exp_pos_neg-exp_pos_pos}.
\begin{figure}[H]
	\centering
	\includegraphics[width=6cm]{ 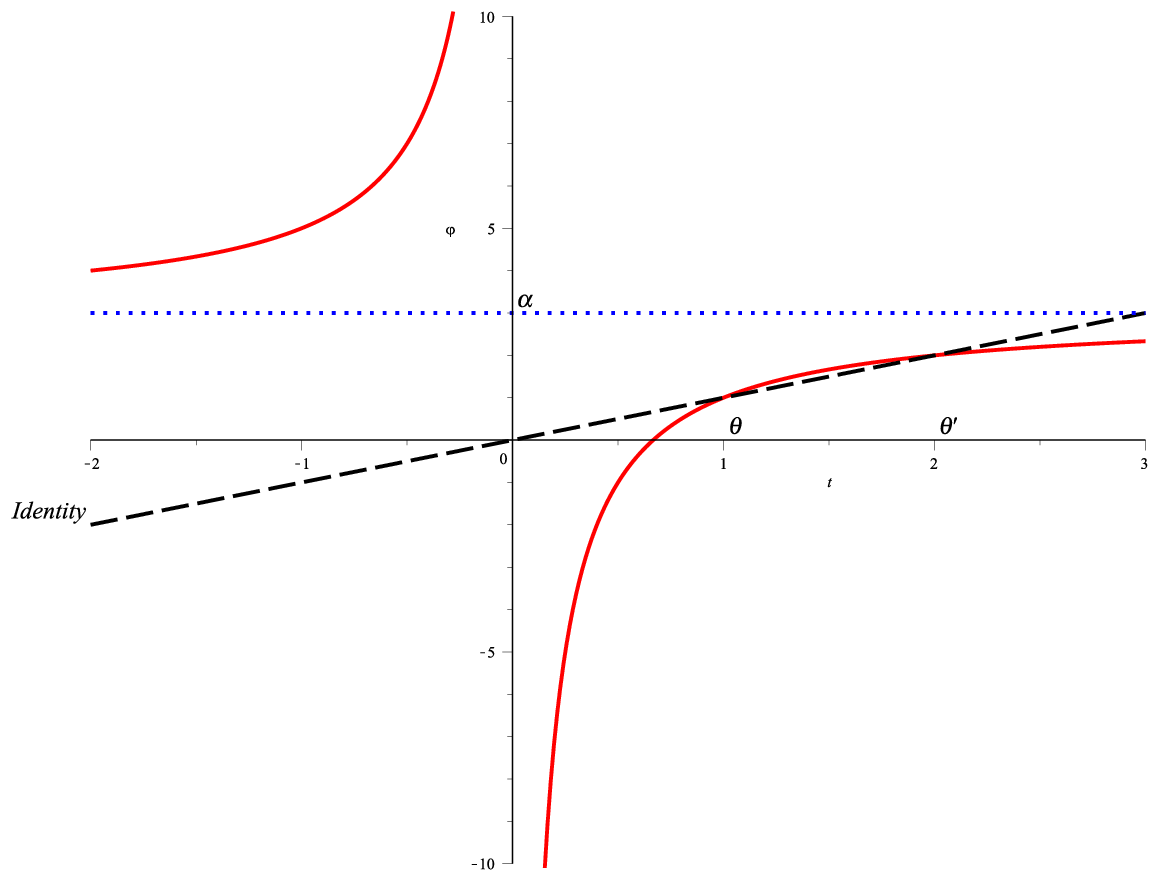}\quad \includegraphics[width=6cm]{ 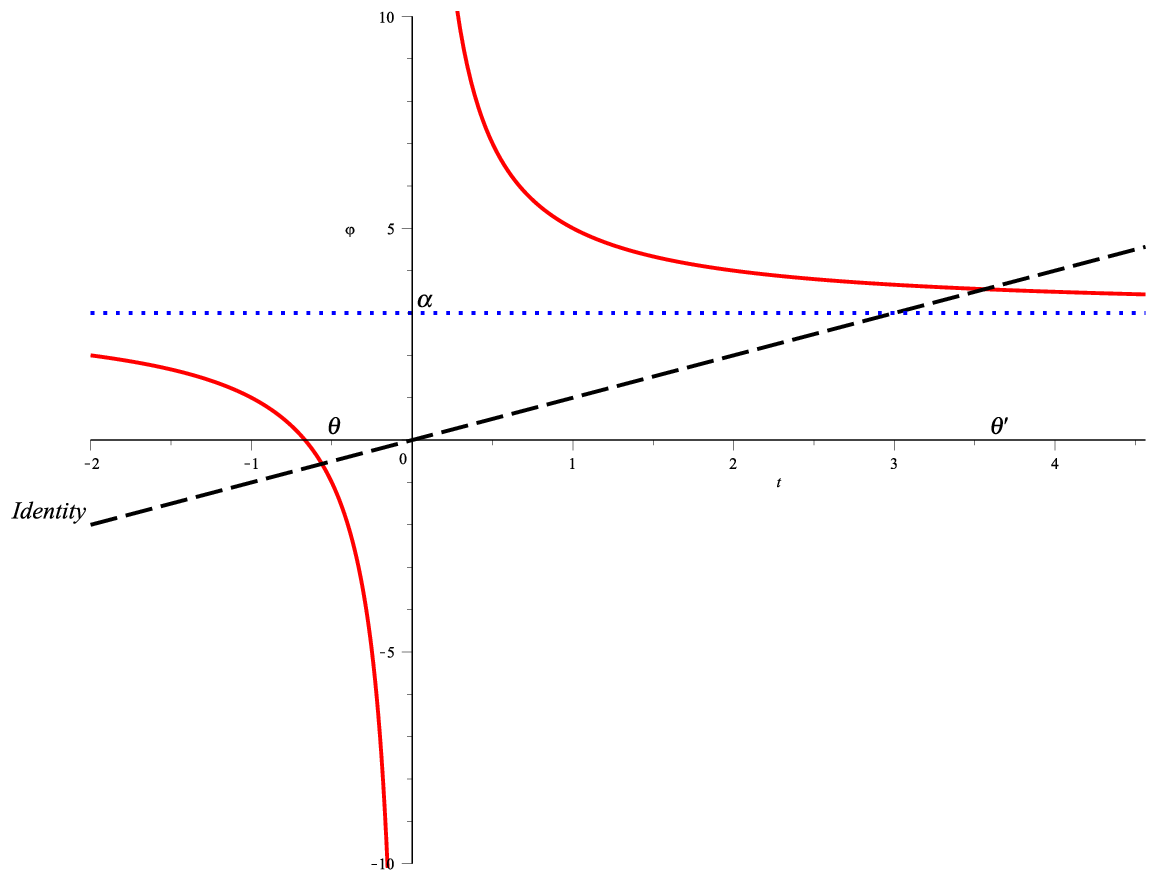}
	\caption{ Dynamics of $\varphi(t)=3 +\frac{2}{t}, \; t \neq 0$  (left)  and  $\varphi(t)=3 -\frac{2}{t}, \; t \neq 0$ (right). }\label{fig:exp_pos_neg-exp_pos_pos}
\end{figure}
Recalling the symmetry property of the solutions from Lemma~\ref{lem:symmetry} , Case (4), $\alpha>0$, $\gamma<0$, is equivalent to Case (2),  $\alpha<0$, $\gamma<0$, by the symmetry $x \to -x$. Analogously, Case (5), $\alpha>0$, $\gamma>0$, is equivalent to Case (3),  $\alpha<0$, $\gamma>0$, by the symmetry $x \to -x$.
All claims in Proposition~\ref{prop: dyn neg neg Delta posit} and Proposition~\ref{prop: dyn neg pos Delta posit} are still valid, but in reversal.

\color{black}

\subsection{Periodic behavior}
Type 3 solutions for $\Delta<0$ are
$$x_{j}= \rho\, \left(\cos(\phi)  - \sin(\phi) \tan(j\phi +\omega) \right),$$
where $\rho=\sqrt{-\gamma},$ $\phi=\arctan\left(\frac{\sqrt{-\alpha^2 - 4 \gamma}}{\alpha}\right)$ and  $\omega \in [0,\,2\pi)$ is defined by the initial point $x_1$. We know that the tangent function is periodic with period $\pi$ but $\omega$ may be not 0. In this case we have a phase translation and the period is $P=\frac{\pi}{\phi}$.

\begin{exemplo} \label{exemplo Laplacian average}
  Let $\alpha=\frac{2}{n}$, $n\geq 3$, and $\gamma=-1$ be fixed numbers. We consider the associated recursion
$$x_{j+1}= \varphi(x_{j})= \alpha + \frac{\gamma}{x_{j}}=\frac{2}{n} -\frac{1}{x_{j}}$$
and $x_{1}=1-(2-\frac{2}{n})=-1+\frac{2}{n}<0$. Computing $\Delta= \alpha^2 + 4 \gamma= \frac{2}{n}^2-4=\frac{4}{n^2}-4<0$. Therefore, we have a Type 3 recursion where
\begin{itemize}
  \item $\rho=\sqrt{-\gamma}=1$;
  \item $\phi=\arctan\left(\frac{\sqrt{-\alpha^2 - 4 \gamma}}{\alpha}\right)=\arctan\left(\frac{\sqrt{4-\frac{4}{n^2} }}{\frac{2}{n}}\right)=\arctan\left(\sqrt{n^2 -1}\right)\in (0,\; \frac{\pi}{2})$;
  \item $\omega=-\phi +\arctan\left(\cot(\phi)-\frac{x_1}{\rho}\csc(\phi)\right) =$\\

      $-\arctan\left(\sqrt{n^2 -1}\right) +\arctan\left(\cot(\phi)-\frac{-1+\frac{2}{n}}{1}\csc(\phi)\right).$

From the identity $\tan(\phi)=\sqrt{n^2 -1}$ we obtain $\csc(\phi)=\frac{n}{\sqrt{n^2 -1}}$ and $\cot(\phi)=\frac{1}{\sqrt{n^2 -1}}$ thus
      $$\omega=-\arctan\left(\sqrt{n^2 -1}\right) +\arctan\left(\frac{1}{\sqrt{n^2 -1}}-\frac{-1+\frac{2}{n}}{1}\frac{n}{\sqrt{n^2 -1}}\right)=$$
      $$=-\arctan\left(\sqrt{n^2 -1}\right) +\arctan\left(\frac{n-1}{\sqrt{n^2 -1}}\right)$$
      Using the formula for the sum $\tan(a+b)$ we conclude that $\omega=-\arctan\left(\frac{n-1}{\sqrt{n^2 -1}}\right)$ and substituting in the previous formula we obtain $\omega=-\frac{\phi}{2}$.
  \item The period is $P=\frac{\pi}{\phi}=\frac{\pi}{\arctan\left(\sqrt{n^2 -1}\right)}$.  Notice that, $\arctan\left(\sqrt{n^2 -1}\right) \in (0,\; \frac{\pi}{2})$. Therefore, $P>2$.
\end{itemize}
\begin{figure}[!ht]
  \centering
  \includegraphics[width=11cm]{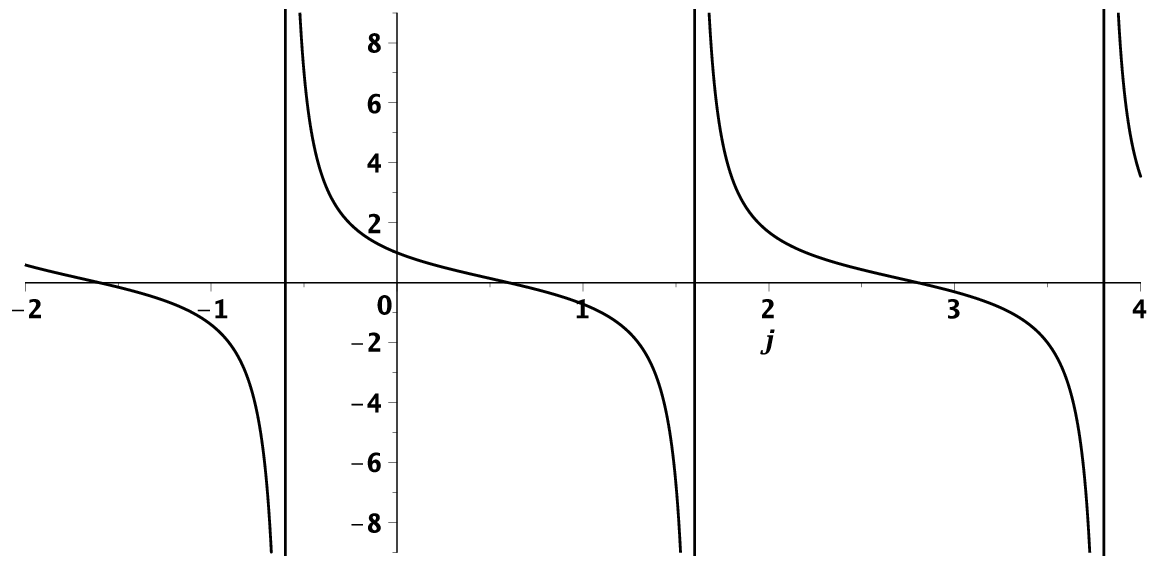}
  \caption{ Function $f(j)=\left(\frac{1}{7}  - \frac{\sqrt{48}}{7} \tan((j-1/2)\phi) \right)$ for $n=7$, $\phi=\arctan(4\sqrt{3})\approx 1.42745$ and the period is $P\approx 2.20084 $.}\label{exem_laplacian_7}
\end{figure}
Thus, the general formula is
$$x_{j}= \left(\frac{1}{n}  - \frac{\sqrt{n^2 -1}}{n} \tan((j-1/2)\phi) \right)$$
because $\csc(\phi)=\frac{1}{n}$ and $\sin(\phi)=\frac{\sqrt{n^2 -1}}{n}$.
\end{exemplo}

\subsection{Dependence on a parameter}
The explicit solution of $ x_{j+1}= \varphi(x_{j}), \; j \geq 1$ which is defined for any $j \in \mathbb{R}$ by $x_{j}= f(j)$ for $j \geq 1$, provided by Theorem~\ref{sol general rational difference equation}, may have same dependence on a given parameter. In the Example~\ref{exemplo Laplacian average} we have that
\begin{equation}\label{eq rec depend param}
   \begin{cases}
     x_{1}=-1+\frac{2}{n}  \\
     x_{j+1}= \frac{2}{n} -\frac{1}{x_{j}}
   \end{cases}
\end{equation}
is depending on the parameter $n\geq 3$, both in the initial point and in the equation. Although, for each value of $n$ the solution presents a similar behavior as we see in the Figure~\ref{exem_laplacian_7}.

For instance, the period  $P(n)=\frac{\pi}{\arctan\left(\sqrt{n^2 -1}\right)}$ depends analytically on $n$. A direct computation shows that $\frac{dP}{dn} <0$ and $\displaystyle\lim_{n \to \infty}P =2$. Therefore, the period is decreasing monotonously to 2.  Varying $n \in[3,\, \infty)$ one can deduce several properties of the class of recursions \eqref{eq rec depend param}.

\section{Applications on limit points}\label{sec:limit}

A real number $r$ is said to be a limit point of the spectral radii of graphs if there exists a sequence $\{G_k\}$ of graphs such that
$$\rho(G_i) \not= \rho(G_j ), \;i \not = j \mbox{ and }\lim_{k\rightarrow \infty} \rho(G_k) = r,$$
where $\rho(G)$ is the spectral radius (or index) of the graph $G$.

A landmark paper in this area is due to J. Shearer \cite{SHEARER1989} who proved that any real number $\lambda  \geq \sqrt{2+\sqrt{5}}$  is a limit point of the spectral radii of graphs. Precisely, there exists an infinite sequence of graphs $G_k, ~k=1,2,\ldots$, whose spectral radius $\rho(G_1)< \cdots <\rho(G_k) < \lambda $ is an increasing  sequence and $\displaystyle\lim_{k \rightarrow \infty } \rho(G_k) =\lambda$.

In general, the techniques used to prove that a real number is a limit point are intricate. By way of an example, we show here that our results may be potentially applicable for finding limit points of spectral radii of graphs.

Consider the starlike tree $T_{l,m,n}$ illustrated in Figure \ref{fig:star} composed by paths $P_l, P_m$ and $P_n$ having one end point each joined by an edge to an isolated vertex.

\begin{figure}[H]
\begin{center}
\begin{tikzpicture}
[scale=1.6,auto=left,every node/.style={circle,scale=0.7}]
\foreach \i in {0,1,2,3}{
     \node[draw,circle,fill=black,scale=.4] (\i) at (\i, 1){};}

\draw (0) -- (1);
\draw (2) --(3);
\draw[dotted] (2) -- (1);

\foreach \i in {0,1,2,3,4}{
     \node[draw,circle,fill=black,scale=.4] (\i) at (\i, 0){};}

\draw (0) -- (1)--(2);
\draw (4) --(3);
\draw[dotted] (2) -- (3);

\foreach \i in {0,1,2,3,4,5}{
     \node[draw,circle,fill=black,scale=.4] (\i) at (\i, -1){};}

\draw (0) -- (1)--(2);
\draw (5)--(4) --(3);
\draw[dotted] (2) -- (3);

\node[draw,circle,fill=black,scale=.4] (0) at (-1.5, 0){};

\draw (0) -- (0,0);\draw (0) -- (0,-1); \draw (0) -- (0,1);
\draw [decorate,decoration={brace,mirror,amplitude=4pt},xshift=0.4pt,yshift=-0.4pt]
(-.2,0.8) -- (3.2,0.8) node [black,midway,yshift=-1cm]{$l$};
\draw [decorate,decoration={brace,mirror,amplitude=4pt},xshift=0.4pt,yshift=-0.4pt]
(-.2,-0.2) -- (4.2,-.2) node [black,midway,yshift=-1cm]{$m$};

\draw [decorate,decoration={brace,mirror,amplitude=4pt},xshift=0.4pt,yshift=-0.4pt]
(-.2,-1.2) -- (5.2,-1.2) node [black,midway,yshift=-1cm]{$n$};

\end{tikzpicture}

\end{center}
\caption{Trees $T_{l,m,n}$}\label{fig:star}
\end{figure}

\begin{thm} [\cite{CIOABA2010}]\label{thm:cioba}
$$\lim_{n \rightarrow \infty} \rho(T_{1,n,n}) = \sqrt{2+\sqrt{5}}.$$
\end{thm}

\begin{proof} Let $T_{1,n,n}$ for $n \geq 1$. After choosing the root as the unique vertex of degree three and enumerating the vertices from the leafs to the root, we see that applying the J-T algorithm, we have the recurrences given by $z_1=-\lambda_n$ and $z_i=-\lambda_n-1/z_{i-1}$, which is the sequence of Example \ref{exemplo Francesco}. For a fixed value of $n$, Figure \ref{fig:star2} illustrates the value in each vertex.

\begin{figure}[H]
\begin{center}
\begin{tikzpicture}
[scale=2,auto=left,every node/.style={circle,scale=0.7}]
\foreach \i in {0}{
     \node[draw,circle,fill=black,scale=.4,label=right:$z_{1}$] (\i) at (\i, 1){};}

\foreach \i in {0,1,2,3}{
     \node[draw,circle,fill=black,scale=.4] (\i) at (\i, 0){};}
     \node[label=above:$z_{1}$] at (3, 0){};
     \node[label=above:$z_{2}$] at (2, 0){};
     \node[label=above:$z_{n}$] at (0, 0){};

\draw (0) -- (1);
\draw (2) --(3);
\draw[dotted] (2) -- (1);

\foreach \i in {0,1,2,3}{
     \node[draw,circle,fill=black,scale=.4] (\i) at (\i, -1){};}
     \node[label=above:$z_{1}$] at (3, -1){};
     \node[label=above:$z_{2}$] at (2, -1){};
     \node[label=above:$z_{n}$] at (0, -1){};
\draw (0) -- (1);
\draw (2) --(3);
\draw[dotted] (2) -- (1);

\node[draw,circle,fill=black,scale=.4,label=left:$a(v)$] (0) at (-1.5, 0){};

\draw (0) -- (0,0);\draw (0) -- (0,-1); \draw (0) -- (0,1);
\end{tikzpicture}
\end{center}
\caption{Trees $T_{1,n,n}$ with labels.}\label{fig:star2}
\end{figure}

Here we consider that $\lambda_n=\rho(T_{1,n,n})$ is the spectral radius of $T_{1,n,n}$. Hence all the values $z_i$ are negative (as predicted by the theory), for $n\geq 3$. The value $a(v) = -\lambda_n- \frac{1}{z_1}-\frac{2}{z_n}=0$, because $\lambda_n$ is an eigenvalue.

We notice that $\lambda_n$ is an increasing sequence as $T_{1,n,n}$ is a proper subgraph of $T_{1,n+1,n+1}$. Moreover $\lambda_n < \frac{3}{\sqrt{3-1}} \approx 2.12$ (see, for example,  \cite{Lepovic2001}). Hence $\lambda_n$ converges to some $\lambda_0$. We recall that $z_{j}=\theta_{n} - \frac{\theta_{n}^{-1} -\theta_{n}}{\left(\theta_{n}^2\right)^j -1}, \; j \geq 1$ and  $\theta_n= \frac{-\lambda_{n}-\sqrt{\lambda_{n}^2-4}}{2}$ thus $z_{n}=\theta_{n} - \frac{\theta_{n}^{-1} -\theta_{n}}{\left(\theta_{n}^2\right)^n -1}$.

Now, if $\left(\theta_{n}^2\right)^n\to \infty$, then $z_n \rightarrow \theta_0= \frac{-\lambda_0-\sqrt{\lambda_0^2-4}}{2}$. In order to see that $ \left(\theta_{n}^2\right)^n\to \infty$, we recall that $\theta=\theta(\lambda)$  is a decreasing function for $\lambda \geq 2=\lambda_2=\rho(T_{1,2,2})$ meaning that, uniformly  for $n \geq 3$, $\left(\theta_{n}^2\right)^n > \left(\theta_{3}^2\right)^n \to \infty$ because $\theta_{3}<-1=\theta_{2}$.

We therefore have  $ \lambda_n \rightarrow \lambda_0$, meaning that $\lambda_{0} - \frac{1}{z_1}-\frac{2}{\theta_0}=0$,
which is equivalent to
$$-{\frac {{\lambda_0}^{2}\sqrt {{\lambda_0}^{2}-4}+{\lambda_0}^{3}-\sqrt {{
\lambda_0}^{2}-4}-5\,\lambda_0}{\lambda_0\, \left( \lambda_0+\sqrt {{\lambda_0}^{2}-4} \right) }} =0.
$$
Solving this gives $\lambda_0= \sqrt{2+\sqrt{5}}$.
\end{proof}

The concept of limit point applies to the Laplacian matrix as well. The Laplacian spectral radius of the tree $T_{1,n,n}$ is denoted by $\mu(T_{1,n,n})$

\begin{thm} [\cite{Guo-2008}]\label{thm:guo}
$$\lim_{n \rightarrow \infty} \mu(T_{1,n,n}) = 2+\epsilon \approx 4.382975767,$$
where $\epsilon$ is the real root of $x^3-4x-4$.
\end{thm}
\begin{proof} We denote by $\lambda_n = \mu(T_{1,n,n})$ for a given fixed $n$.  We fisrt notice that $$ 4 < \lambda_n \leq 5.$$
The lower bound follows from \cite{Anderson1985} where it is shown that the Laplacian spectral radius is bounded by $\max \{ d(v) + d(u) | u,v \mbox{ are adjacent} \}$, where $d(v)$ is the degree of the vertex $v$. The upper bound follows from \cite{Grone1994}, where it is shown that the Laplacian spectral radius is greater than or equal the largest degree + 1, with equality if and only if there is a universal vertex.

It follows that the sequence $\lambda_n$ converges to $\lambda_0$, since is limited and is monotonic increasing, as $T_{1,n,n}$ is a proper subgraph of $T_{1,n+1,n+1}$.

The same application of the J-T algorithm to $T_{1,n,n}$ with $\lambda_n= \mu(T_{1,n,n})$, gives the sequence $d_1=1-\lambda_n$ and $d_i=2-\lambda_n -\frac{1}{d_{i-1}}$, for $i > 1$. In the language of Theorem~\ref{sol general rational difference equation}, we have $\alpha = 2- \lambda_n$ and  $\gamma=-1$, leading to $\Delta = \alpha^2 + 4\gamma=\lambda_n^2-4 \lambda_{n} > 0$, as $\lambda_n > 4$. The fixed points are $\theta_{n}=\frac{2-\lambda_n-\sqrt{\lambda_n^2-4\lambda_n}}{2}$ and $\frac{1}{\theta_{n}}$. Now, since $\lambda_n >4$, we have that $\lambda_n < \theta_{n}$, then we know that $d_i \rightarrow \theta_{n}$, increasingly. The value $a(v)$ is now $a(v)= 3 -\lambda_n - 1/d_1 -2/d_n=0$, since $\lambda_n$ is an eigenvalue.

We have $\mu(T_{1,n,n}) \rightarrow \lambda_0$ and, by a similar argument given in Theorem~\ref{thm:cioba}, $\theta_n \rightarrow \theta_{0}$. Therefore $\mu(T_{1,n,n}) \rightarrow \lambda_0$, with
$\lambda_0 = 3-\lambda_0-\frac{1}{d_1}-\frac{2}{\theta_{0}} =3 - \lambda_0 -\frac{1}{1-\lambda_0}-\frac{2}{\theta_{0}}$ and $\theta_0=\frac{2-\lambda_0-\sqrt{\lambda_0^2-4\lambda_0}}{2}$, implying $-\lambda_0 -\frac{1}{1-\lambda_0}-\frac{2}{\theta_0}=0$, which is equivalent to
$${\lambda_0}^{3}+ \left( \sqrt {\lambda_0\, \left( \lambda_0-4 \right) }-6
 \right) {\lambda_0}^{2}+ \left( -4\,\sqrt {\lambda_0\, \left( \lambda_0-4
 \right) }+6 \right) \lambda_0+2\,\sqrt {\lambda_0\, \left( \lambda_0-4
 \right) }=0.$$

Solving this equation gives
$$\frac{\sqrt [3]{54+6\,\sqrt {33}}}{3}+\,{\frac {4}{\sqrt [3]{54+6\,\sqrt {33}}}}+2.$$
To finish the proof, we let $\epsilon = \frac{\sqrt [3]{54+6\,\sqrt {33}}}{3}+\,{\frac {4}{\sqrt [3]{54+6\,\sqrt {33}}}}\approx  2.382975767$. It is easy to check that $\epsilon$ is the only real root of $x^3 - 4x -4$.

\end{proof}

\section{A further look on parameter dependence}\label{sec:param}

In this section we emphasize the parameter dependence of the recurrences and apply to some classes of problems.  The general idea is to study properties of a family of recursions
\begin{equation}\label{eq rec depend param general}
   \begin{cases}
     x_{1}=h(s)  \\
     x_{j+1}= \alpha(s) +\frac{\gamma(s)}{x_{j}}
   \end{cases}
\end{equation}
where $s$ is a real parameter varying in the set $I\subset \mathbb{R}$ and $h(s),\alpha(s), \gamma(s)$ are continuous functions of $s$ (in general we assume to have good differential properties).

We think the dependence on the parameter in the following way: we fix a  parameter $s$ and solve \eqref{eq rec depend param general} with respect to $j$ obtaining $x_j$. As the solution depends on the chosen $s$, as $j$ iterates, $x_j$ is actually a function of $s$, denoted by $x_j(s)$. Now we want to understand the behaviour of $x_{j}(s)$ for all $s \in I$. We represent this dependence in the next table, where $I=\mathbb{N}$,
\[\left|
  \begin{array}{cccccc}
     s\setminus j & 1 & 2 & 3 & 4 & \cdots \\
    1 & x_{1}(1) & x_{2}(1) & x_{3}(1)& x_{4}(1) & \cdots \\
    2 & x_{1}(2) & x_{2}(2) & x_{3}(2) & x_{4}(2) & \cdots \\
    3 & x_{1}(3) & x_{2}(3) & x_{3}(3) & x_{4}(3) & \cdots \\
    4 & x_{1}(4) & x_{2}(4) & x_{3}(4) & x_{4}(4) & \cdots \\
    \vdots & \vdots & \vdots & \vdots & \vdots & \ddots \\
  \end{array}
\right|
\]
which is a pictorial representation of  $x_{j}(s)$.  Suppose we are interested in understanding the behavior of $x_j(s)$, when $j$ is fixed. This means we look downwards in the column $j$. Perhaps a more standard technique would be, for a fixed value of $s$, study the function $x_j(s)$ as $j$ varies. This means we look at the row $j$ from left to right. We recall here that in general we are interested in the variation of $x_j(s)$ when $j$ is a natural number, but our formulas hold for any real $j$ and hence we may take advantage of the analytical properties of $x_j(s)$ (see a simple illustration in Example \ref{ex:extended}).

By Theorem~\ref{sol general rational difference equation} all the parameters involved in the solutions are elementary functions (rational, trigonometric, inverse trigonometric, square, etc) of the parameters $x_1$, $\alpha$ and $\gamma$. Therefore, we can assume that all the iterates $x_j(s)$ are smooth (say $C^\infty$) functions of $s$, in its respective domains with vertical asymptotes in some points, provided that $h(s)$, $\alpha(s)$ and $\gamma(s)$ have this property.

There are infinitely many possibilities of formulations of this problem. We will consider a specific example that may be seen as an application to important problems in spectral graph theory.

\subsection{Laplacian eigenvalues smaller than the average} In a recent publication \cite{Jacobs2021}, it is shown that the number of eigenvalues of any tree (with $n$ vertices) smaller than the average degree $d=2- \frac{2}{n}$ is at least $ \lceil \frac{n}{2}\rceil$. The main technique used was the analysis of the recurrence associated to the application of J-T algorithm to a path having a number $r$ of pendant $P_2$ (see Fig. \ref{newseq_r}).

Applying the J-T algorithm to the Laplacian matrix  to locate $\alpha=d$  we obtain, in each extremal vertex of the pendant paths $P_{x_{1}x_{2}}$ the value
$$x_{1}=1-d=-1+\frac{2}{n}<0$$
 and the next value is
$$x_{2}=2-d -\frac{1}{x_{1}}=\frac{2}{n}-\frac{1}{x_{1}}>1.$$
From this values we follow the processing to a root obtaining
$$b_{1}=r+1-d-\frac{r}{x_{2}}= x_{1} + r\left(1- \frac{1}{x_{2}}\right),$$
and the rest of the values are given by the recursion
\begin{equation}\label{depend param case 2}
   \begin{cases}
     b_{1}=x_{1} + r\left(1- \frac{1}{x_{2}}\right)  \\
     b_{j+1}= \frac{2}{n} -\frac{1}{b_{j}}
   \end{cases}
\end{equation}
for $n \in[3,\; \infty)$.

\begin{figure}[H]
  \centering
  \includegraphics[width=8cm]{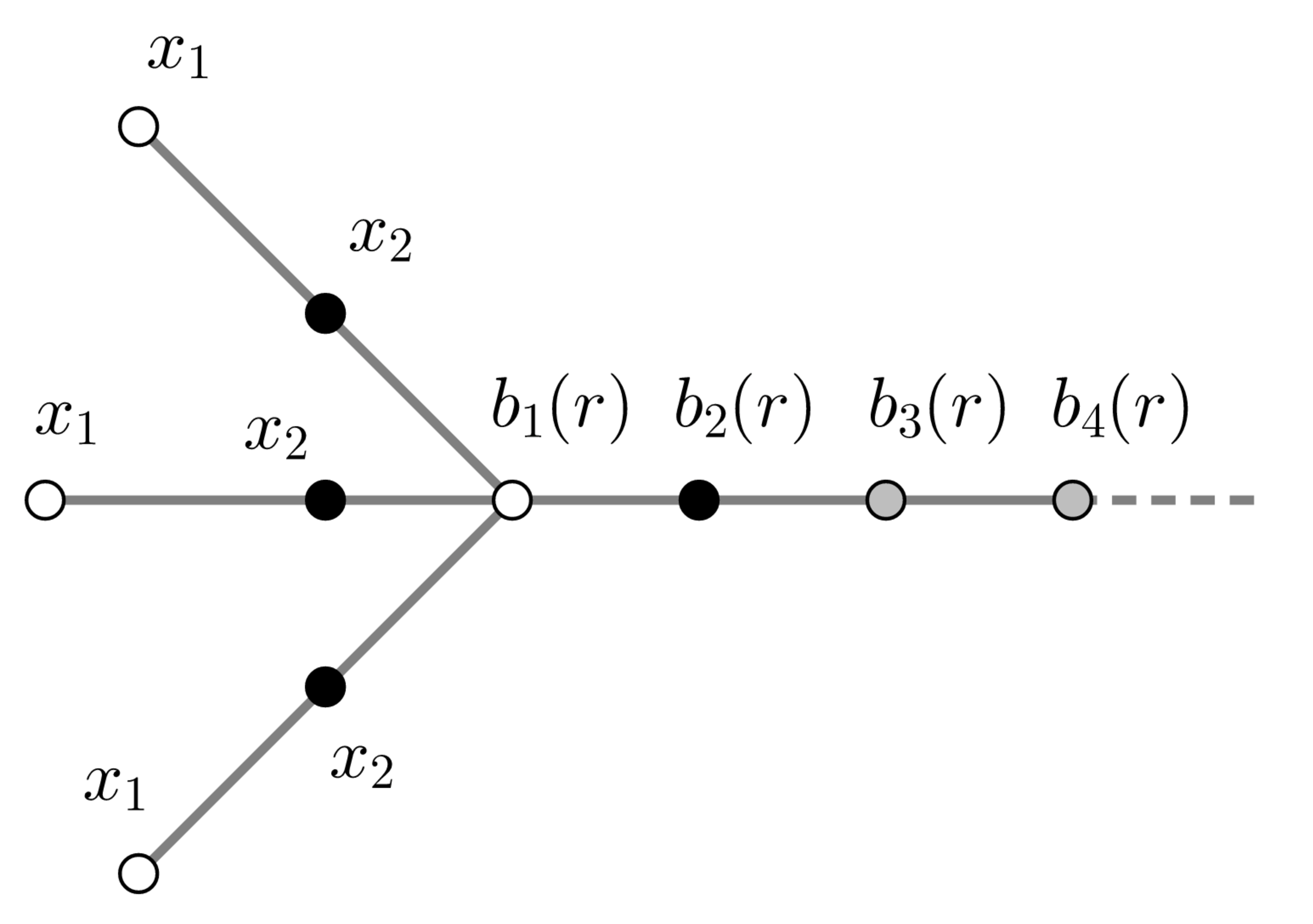}
  \caption{White dots in the vertices means \texttt{Diagonalize($T,-2+\frac{2}{n}$)} produces a \textbf{negative} value,
while black vertices  means a positive value and light gray where we do not know the precise sign (e.g. $b_3(r)$, $b_4(r)$, etc).}\label{newseq_r}
\end{figure}

In what follows, we obtain properties of the signs at the vertices of the paths that were not presented in \cite{Jacobs2021}.

\subsubsection{Initial value}
Our first concern is the dependence of the initial condition with respect to $r$. The next lemma may be found in \cite{Jacobs2021}. For completeness, we present a proof.

\begin{lema}\label{depend   init  cond  to r}
Let $n \geq 8$. If $0 \leq r \leq \lfloor\frac{n}{4}\rfloor$ then
$b_{1}(r)<0$.
\end{lema}
\begin{proof}
Assume $n \geq 8$.
Recall that $b_{1}(r)=x_{1} + r(1- \frac{1}{x_{2}})$.
On the non-negative reals, define the linear function into $\mathbb{R}$
$$
g(r) = x_{1} + r \left(1- \frac{1}{x_{2}}\right).
$$
Then $g(0)=x_{1}<0$.
Since $g'(r) = (1- \frac{1}{x_{2}} )>0$, it is increasing.
By continuity there is a unique point $g(r_0)=0$.
Solving for $r_0$ we have
$r_0=\frac{-x_{1}}{1- \frac{1}{x_{2}}} > 0$.
Since $x_1 = -\frac{n-2}{n}$
and
$x_2 = \frac{n^2 + 2n - 4}{n^2 - 2n}$,
$r_0$ depends rationally
on $n$
$$
     r_0=\,{\frac { \left( n-2 \right)  \left( {n}^{2}+2\,n-4 \right) }{4n
 \left( n-1 \right) }}.
$$
Therefore $b_{1}(r)<0$ if and only if $r\leq \lfloor r_0 \rfloor$.
We claim that
$\frac{n}{4} < r_0$ for
$n \geq 7$.
Indeed, the inequality $\frac{n}{4} < r_0$ can be simplified
to $0 < n^2 - 8n + 8$ whose largest root is $4 + 2\sqrt{2} \approx 6.8$.
\end{proof}

\subsubsection{General formula}
The recursion \eqref{depend param case 2} is the same we used in the Example~\ref{exemplo Laplacian average} hence we have a Type 3 recursion, $b_{j}=\rho\, \left(\cos(\phi)  - \sin(\phi) \tan(j\phi +\omega) \right)$ where
\begin{itemize}
  \item $\rho=1$;
  \item $\phi=\arctan\left(\sqrt{n^2 -1}\right)\in (0,\; \frac{\pi}{2})$;
    \item The period is $P=\frac{\pi}{\arctan\left(\sqrt{n^2 -1}\right)}>2$ because $\arctan\left(\sqrt{n^2 -1}\right) \in (0,\; \frac{\pi}{2})$
\end{itemize}

For the phase translation $\omega$ the situation is different because it depends on $r$. Denote  $\omega_{r}$ the phase translation associated with the initial condition $b_{1}=x_{1} + r\left(1- \frac{1}{x_{2}}\right)$ then
$$\omega_{r}=-\phi +\arctan\left(\cot(\phi)-\frac{b_1}{\rho}\csc(\phi)\right).$$
From the identity $\tan(\phi)=\sqrt{n^2 -1}$ we obtain $\csc(\phi)=\frac{n}{\sqrt{n^2 -1}}$ and $\cot(\phi)=\frac{1}{\sqrt{n^2 -1}}$ thus
$$\omega_{r}=-\arctan\left(\sqrt{n^2 -1}\right) +\arctan\left(\frac{1}{\sqrt{n^2 -1}}- b_{1}  \frac{n}{\sqrt{n^2 -1}}\right)=$$
\begin{equation}\label{form wr}
  = \arctan\left(\frac{1-n b_{1} }{\sqrt{n^2 -1}}\right)-\arctan\left(\sqrt{n^2 -1}\right).
\end{equation}
For each fixed $n$ the phase translation $\omega_{r}$ depends differentially on $r$. A straightforward computation shows that
$$\frac{d\omega_{r}}{dr}=\frac{1}{1+\left( \frac{1-n \, b_{1}}{\sqrt{n^2 -1}} \right)^{2}}\frac{ -n\left(1- \frac{1}{x_{2}}\right)}{\sqrt{n^2 -1}} =$$
$$={\frac {-2\,\sqrt {{n}^{2}-1} \left( {n}^{2}+2\,n-4 \right) }{8\,n
 \left( n-1 \right) {r}^{2}-4\, \left( n-1 \right)  \left( {n}^{2}+2\,
n-4 \right) r+ \left( {n}^{2}+2\,n-4 \right) ^{2}}}
<0,$$
therefore $\omega_{r}$ is a strictly decreasing function of $r$. Additionally $\displaystyle\lim_{n \to \infty} \omega_{r} =  -\frac{\pi}{2}+ \arctan\left(-\lim_{n \to \infty} b_1\right)=-\frac{\pi}{4}$, because $\displaystyle \lim_{n \to \infty} b_1(r) =-1$.
Thus \begin{equation}\label{interval wr}
  -\frac{\pi}{2} < \omega_{r} < -\frac{\pi}{4}.
\end{equation}

\subsubsection{Periodicity and sign change}
As we have observed before, the function $j \to f(j):=b_{j}=\left(\cos(\phi)  - \sin(\phi) \tan(j\phi +\omega_{r})\right)$ is actually defined for $j \in \mathbb{R}$, except for vertical asymptotes. It is periodic with period $P>2$ and piecewise decreasing in each interval of the period. As $\displaystyle\lim_{n \to \infty} P=  2$, the behaviour of $b_j$ is very similar to a $\mod2$-periodic function.

From Lemma~\ref{depend   init  cond  to r} we know that $b_1<0$, provided that $1\leq r \leq \lfloor\frac{n}{4}\rfloor$. Using the recurrence formula, we have $b_2=\frac{2}{n} - \frac{1}{b_1}>0$. The main goal of this section is to discover how large is this alternating sign pattern, that is how large is $k$ so that $b_{2k+1}<0$ and consequently $b_{2k+2}>0$. This change in the pattern will happen when we find a $k_{0} >0$ such that $b_{2k_0+1}>0$ (notice that for $k_{0}=0$ we obtain $b_{2k_0+1} =b_1<0$).

More formally we define
$$k_{0}:=\left\{\begin{array}{cc}
          \max ~ \{k~ |~ b_{2k+1}<0 \,\}, & \mbox{ if such $k$ exists}\\
          \infty,  & \mbox{ otherwise. }\\
          \end{array}\right.
$$
Obviously $k_{0} \geq 1$. A remarkable fact is that, if $k_{0}<\infty$ then $b_{2k_0+2}>0$ and $b_{2k_0+3}>0$ is the first consecutive pair of positive numbers in the sequence $b_{j}=b_{j}(r)$.

\begin{figure}[H]
  \centering
  \includegraphics[width=12cm]{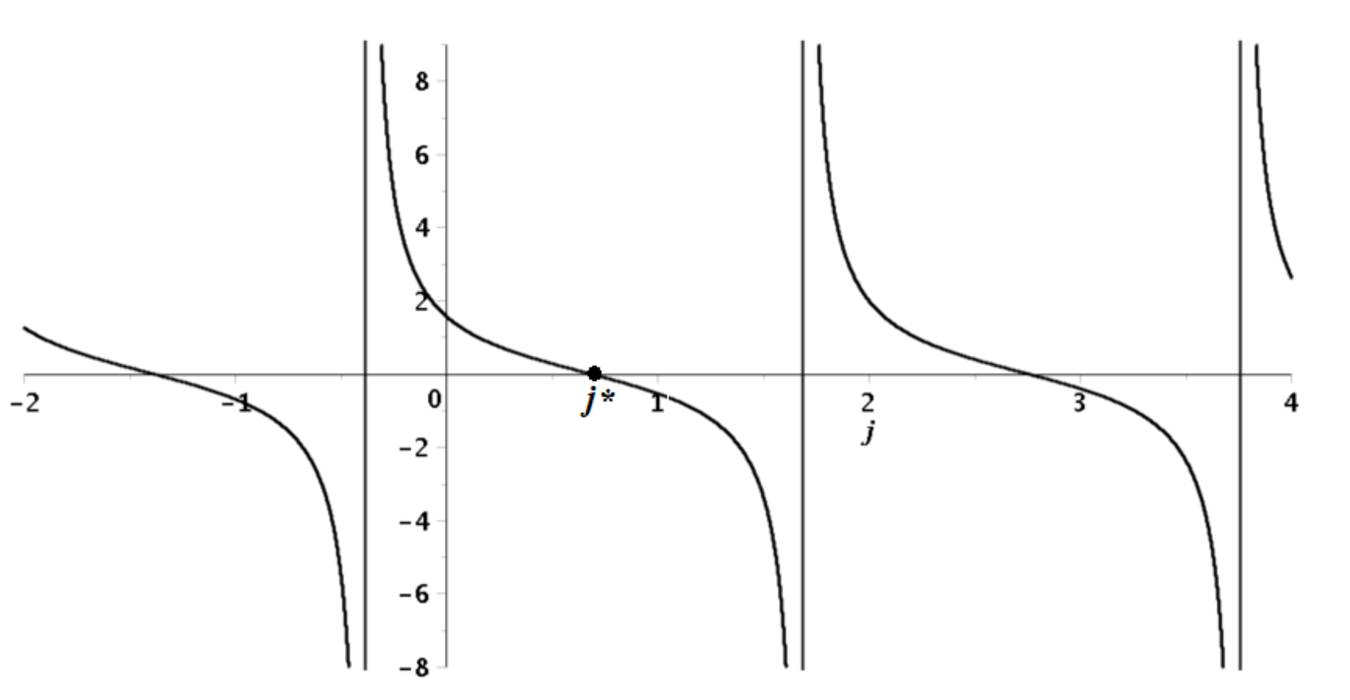}
  \caption{ Extended solution $j \to f(j)=b_{j}$ for $r=2$ and $n=19$. In this case $P = 2.069368956$. The first positive root is $j^{*}=0.6867$.}\label{solut_r2n19}
\end{figure}

\begin{thm}\label{thm:old_mcs_r}
   Let $n\geq 8$ and $1\leq r \leq \lfloor\frac{n}{4}\rfloor$. Then, $k_{0}<\infty$ and it is given by
   \begin{equation}\label{eq:estima_old_mcs_r}
      k_{0}=\left\lfloor \frac{1}{P-2}+ \frac{\omega_{r} -\arctan(\cot(\phi))}{\phi(P-2)}\right\rfloor
   \end{equation}
\end{thm}
\begin{proof}
Our approach is to use the fact that $b_j$ has a period close to $2$ in order to compare different odd indices with $b_1$ which we know to be negative. For $b_{2k+1}$ we compare
$$b_{2k+1}=b_{1+2k -kP+ kP}=b_{1 -(P-2)k + kP}=b_{1 -(P-2)k}.$$
Since  $b_{1 }= b_{1+kP}<0$, $b_{j}$ is decreasing and $-(P-2)k<0$, we conclude that the correspondence $k \to b_{2k+1}$ is strictly increasing while  $1-(P-2)k > j^*$ where $j^*$ is the first positive zero (see Figure~\ref{solut_r2n19}). Stronger than that is the fact that if $k$ is big enough so that $1-(P-2)k < j^*$ then we have $b_{2k+1}>0$. This shows that $k_{0}=\left\lfloor \hat{k}\right\rfloor$, where $\hat{k}:=\frac{1-j^*}{P-2}$, in particular $k_{0}<\infty$. In order to estimate $k_{0}$ we are going to introduce a new function
\begin{equation}\label{eq:aux_k_0}
   G(k):=f(j^*)=f(1-(P-2)k)\; k\geq 0,
\end{equation}
where, $f(j)=\left(\cos(\phi)  - \sin(\phi) \tan(j\phi +\omega_{r})\right)$ is the explicit solution of the recurrence. Obviously, $\hat{k}$ is the first positive root of $G$. A simple calculation shows that $G(\hat{k})=0$ is equivalent to $\cos(\phi)  - \sin(\phi) \tan((1-(P-2)\hat{k})\phi +\omega_{r})=0$. Now a (tedious) computation shows that it is equivalent to the existence of some $m \in \mathbb{Z}$ such that
$$\hat{k}=\frac{1}{P-2}+ \frac{\omega_{r} -\arctan(\cot(\phi))}{\phi(P-2)} -m \frac{P}{P-2}.$$

We claim that  $m=0$ is the value that recovers the first positive root. By properties of periodic functions we just need to examine the values $m\in\{-1,0,1\}$. To do that we introduce a new function $H:\mathcal{A} \to \mathbb{R}$ given by
\begin{equation}\label{eq:H_func_k_hat}
H(n,r,m)= \frac{1}{P-2}+ \frac{\omega_{r} -\arctan(\cot(\phi))}{\phi(P-2)} -m \frac{P}{P-2}
\end{equation}
where $m$ is a fixed parameter and $\mathcal{A}:=\{(n,r)\;|\; n \geq 8, \; 1\leq r \leq \lfloor\frac{n}{4}\rfloor \}$. The graphs of Figure~\ref{test_m} below confirm the correct choice $m=0$ because decreasing or increasing $m$ will generate negative or positive roots, respectively, showing that we reach the first positive root in $m=0$, which concludes our proof.
\begin{figure}[H]
  \centering
  \includegraphics[width=15cm]{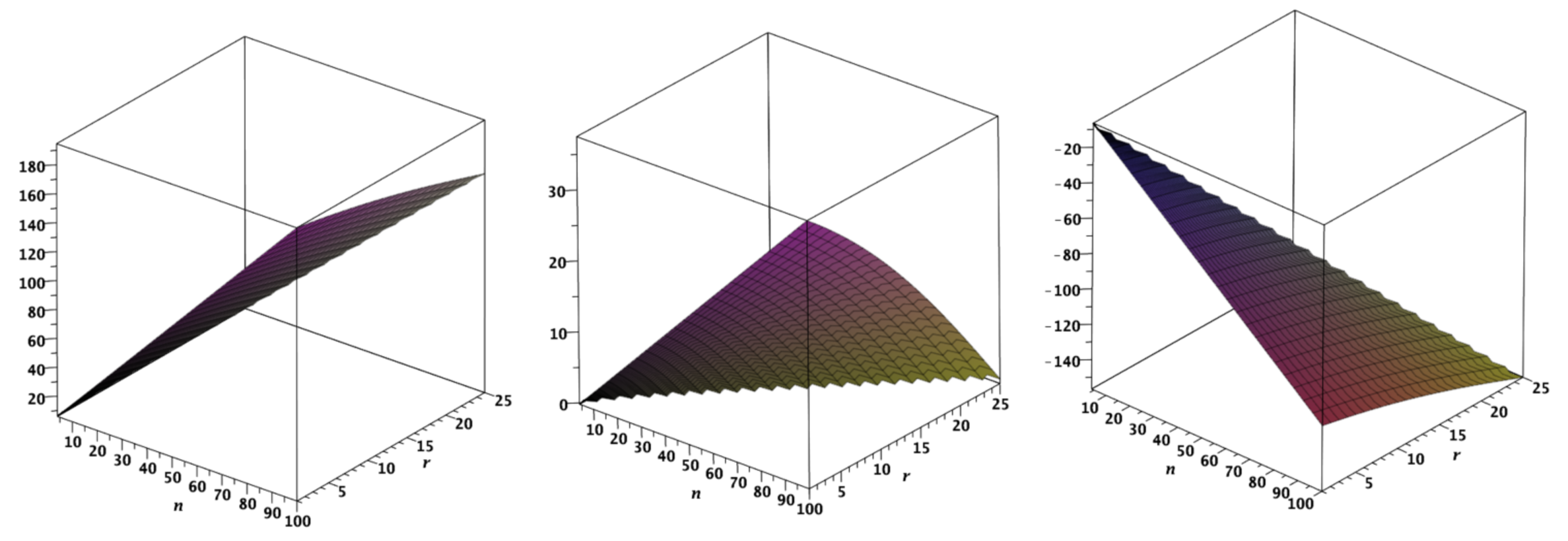}
  \caption{From left to right, the 3d plot of $H(n,r,m)$ for $m\in\{-1,0,1\}$:  $H(n,r,-1)>0\,$, $H(n,r,0)>0,\;$ and $H(n,r,1)<0,\;$.}\label{test_m}
\end{figure}
\end{proof}

\subsubsection{The maximum length of alternating signs $({\rm mlas})$}
First we observe that there is no point in defining a minimum length of alternating signs  in the sequence $b_j(r)$  because it is always 2, since from Lemma~\ref{depend   init  cond  to r} we know that $b_1<0$ and $b_2>0$, provided that $n \geq 8$ and $1\leq r \leq \lfloor\frac{n}{4}\rfloor$. Our main concern is to predict the maximum length where the sign alternance happens. We introduce the definition of the maximum length of alternating signs or ${\rm mlas}$, for short.

\begin{defin}\label{def:mlas}
   Given $n \geq 8$ and $1\leq r \leq \lfloor\frac{n}{4}\rfloor$, consider the sequence $\{b_j(r), \, j\geq 1\}$ given by equation \eqref{depend param case 2}, then we define the maximum length of alternating signs of $b_j(r)$ as being
\begin{equation}\label{mcsr definition}
  {\rm mlas}_{r}(n)=
      2k_r+2, \text{ if } b_{2k+1}<0, \text{ for }0\leq k\leq k_r\text{ and } b_{2(k_{r}+1)+1}>0.
\end{equation}
\end{defin}
We notice that, from Theorem~\ref{thm:old_mcs_r}, ${\rm mlas}_{r}(n)=2k_{0}+2<\infty$ is well defined and, consequently, $b_{2k+1}<0, b_{2k+2}>0$ for $0 \leq k \leq k_r$, represents exactly the maximum index $j$ having only pairs -,+ in the sequence $b_{j}$.

\begin{exemplo}\label{ex:n=19_r=2}
   Consider the example given by Figure~\ref{solut_r2n19}, $r=2$ and $n=19$. In this case $P \approx 2.06$, the first positive root is $j^{*}\approx 0.68$, $\phi\approx 1.51$, $\omega\approx -0.98$ and $k_{0}=\left\lfloor \frac{1}{P-2}+ \frac{\omega_{r} -\arctan(\cot(\phi))}{\phi(P-2)}\right\rfloor= \lfloor 4.51 \rfloor=4$. This means that we must have alternance (-,+) in the sequence elements until $2k_{0}+2= 10$ and $b_{11}\approx 0.05>0$ should be the first odd index which is positive. Indeed, computing the entire sequence we verify that $$\{b_{j}(2)\}=\{-0.53, 1.99, -0.39,
   2.62, -0.27, 3.73, -0.16, 6.25,
   -0.05, 18.39,  0.05,...\},$$ showing the power of the formula given by Theorem~\ref{thm:old_mcs_r}. In this example ${\rm mlas}_{2}(19)=2k_{0}+2=10$.
\end{exemplo}

Despite the fact that Theorem~\ref{thm:old_mcs_r} provides an exact formula for ${\rm mlas}_{r}(n)$, we would like to have a lower bound for the ${\rm mlas}_{r}(n)$ which allows one to have some idea of the size of this number in a general situation without computing $k_{0}$. The next result is remarkable and is obtained from estimates of the trigonometric functions involved in the formula for $k_{0}$.

\begin{thm}\label{mcsr theorem}
  Consider $n \geq 8$ and $1\leq r \leq \lfloor\frac{n}{4}\rfloor$, then a lower bound for the ${\rm mlas}_{r}(n)$ is
$${\rm mlas}_{r}(n) \geq \max\left\{2 \left\lfloor \frac{\pi}{8}(n-2) \right\rfloor - 4(r-1),\;2 \right\}.$$
\end{thm}
\begin{proof}
   We recall that, ${\rm mlas}_{r}(n)=2 k_0 +2 \geq 2$ and  $k_{0}=\left\lfloor \frac{1}{P-2}+ \frac{\omega_{r} -\arctan(\cot(\phi))}{\phi(P-2)}\right\rfloor$. Plotting the functions $F_1:(n,r) \to 2\left(\frac{1}{P-2}+ \frac{\omega_{r} -\arctan(\cot(\phi))}{\phi(P-2)}\right)+2$ (in red) and $F_2:(n,r) \to 2 \left\lfloor \frac{\pi}{8}(n-2) \right\rfloor - 4(r-1)$ (blue) in the domain $\mathcal{A}:=\{(n,r)\;|\; n \geq 8, \; 1\leq r \leq \lfloor\frac{n}{4}\rfloor \}$ one may see that $F_1(n,r)>F_2(n,r)$, We remark that the approximation is much sharper when $r \to 1$.
\begin{figure}[H]
  \centering
  \includegraphics[width=12cm]{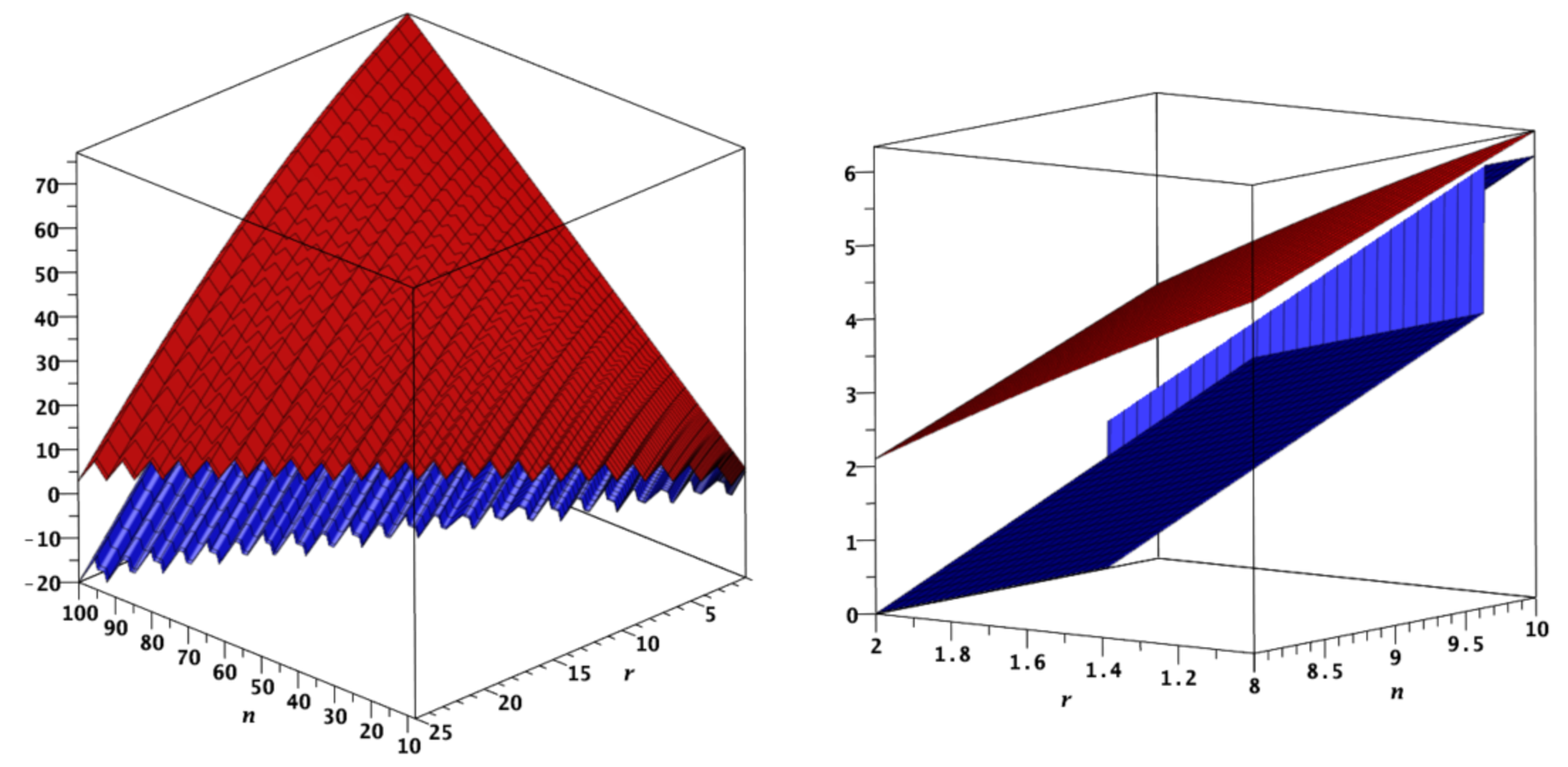}
  \caption{Comparing  ${\rm mlas}_{r}(n)=2 k_0 +2$ and $F_2(n,r)$ for $8 \leq n \leq 100$ (left) and a zoom for $8 \leq n \leq 10$ (right), showing the inequality from a different angle.}\label{fig:test_mlas}
\end{figure}
\end{proof}

\subsubsection{Estimating ${\rm mlas}_{1}(n)$}
A very special case is when $r=1$ because we recover information on the original sequence $x_j$ obtained by using the J-T algorithm applied to an actual pendant path of a given tree. Recall that, for a fixed $r$ we have $b_{1}(r)=x_{1} + r\left(1- \frac{1}{x_{2}}\right)$ and  $b_{j+1}= \frac{2}{n} -\frac{1}{b_{j}}$  where
$x_{1}=1-d=-1+\frac{2}{n}<0 $ and $x_{2}=2-d -\frac{1}{x_{1}}=\frac{2}{n}-\frac{1}{x_{1}}>1$.
Hence $b_{1}(1)=x_{1} + \left(1- \frac{1}{x_{2}}\right)=x_{3}$, analogously~\footnote{It is also possible to consider $r=0$, in this case $b_{j}(0)=x_{j}, \; j\geq1$ because $b_{1}(0)=x_{1} + 0\left(1- \frac{1}{x_{2}}\right)=x_{1}$ and the recurrence is the same}, $b_{j}(1)=x_{j+2}, \; j\geq1$. As the sequence $x_j$ starts with alternate signs (see Figure~\ref{fig:mlas}), we conclude that the maximum length of alternating signs for $x_j$, denoted ${\rm mlas}(n)$ satisfy
\begin{equation}\label{eq:mlas_xj}
  {\rm mlas}(n):={\rm mlas}_{1}(n)+2.
\end{equation}
\begin{figure}[H]
  \centering
  \includegraphics[width=11cm]{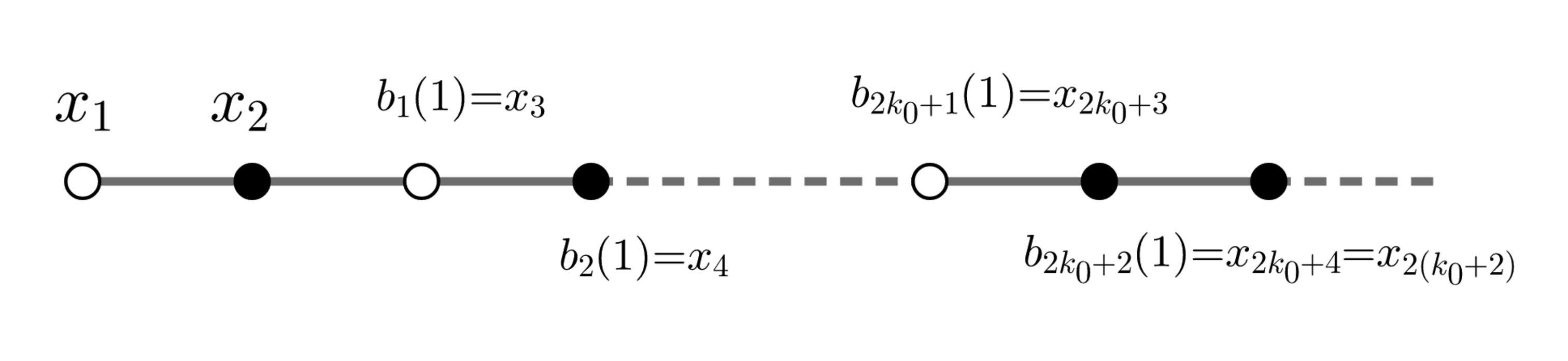}
  \caption{Case $r=1$.}\label{fig:mlas}
\end{figure}
We notice that, from the proof of Theorem~\ref{mcsr theorem}, the estimate for $r=1$ is quite sharp. In particular, for $r=1$ we have, for  $x_{j}$ that
$${\rm mlas}(n) \geq 2 \left\lfloor \frac{\pi}{8}(n-2) \right\rfloor + 2 \approx \frac{3}{4} (n-1).$$
Roughly speaking, if the length of a path is smaller than 75\% of the graph order, we always alternate -,+ which allows us to conclude that 50\% of the eigenvalues associated with the vertices in this path are above/under the average of the Laplacian eigenvalues!

\begin{obs} \label{obs:compare mlas and mlasr}
We would like to conclude this section by pointing out the role of the discounting term $-4(r-1)$ in Theorem~\ref{mcsr theorem}. As the table below illustrates, we can see that:
\begin{itemize}
  \item ${\rm mlas}_{1}(n)$ is the largest possible interval of alternating signs;
  \item ${\rm mlas}_{r}(n)$ actually decreases as a function of $r$;
  \item The lower bound given by Theorem~\ref{mcsr theorem} is sharp for $r=1$ and looses accuracy as $r$ increases. In fact, each time $r$ increases by 1, the estimated value decreases by approximately 4, while the actual values may decrease by 2. It can even be negative, and in such case it is meaningless, because we always have ${\rm mlas}_{r}(n) \geq 2$.
\end{itemize}
In this table we consider $n=183$, a fairly large number, and compare ${\rm mlas}_{r}(183)$ for increasing values of $r$ exhibiting the lower bound and the elements where the sign change. We recall that we must consider only $r\leq \left\lfloor \frac{183}{4} \right\rfloor=45$.
\begin{center}
\begin{tabular}{|c|c|c|c|c|c|}
\hline
      & $_{{\rm mlas}_{r}(183)=2k_{0}+2}$ & $_{2 \left\lfloor \frac{\pi}{8}(n-2) \right\rfloor - 4(r-1)}$ & $b_{2k_{0}+2}$ & $b_{2k_{0}+3}$ \\
      \hline
  r=1 & 142 & 142 & 805.62903 & 0.0096876957  \\
  r=2 & 140 & 138 & 996.23786 & 0.0099251854  \\
  \vdots & \vdots & \vdots  & \vdots  & \vdots   \\
  r=25 & 78 & 46 & 128.14968 & 0.0031255870  \\
  r=26 & 76 & 42 & 1255.7188 & 0.010132605  \\
  r=27&  72&  38&  225.69361&  0.0064999950 \\
  r=28&  68&  34&  128.43487&  0.0031447334 \\
  r=29&  64&  30&  91.942550&  0.0000494238 \\
  r=30&  62&  26&  361.13179&  0.0081597084 \\
  \vdots & \vdots & \vdots  & \vdots  & \vdots   \\
  r=44 & 8 & -30 & 100.17364 & 0.00094629571  \\
  r=45 & 4 & -34 & 96.474616 & 0.00056354082  \\
  \hline
\end{tabular}
\end{center}
As predicted ${\rm mlas}_{r}(183)$ decreases when $r$ increases, sometimes by 2 and, in the worst case 4.  When the pendant star becomes heavier, we loose the sign alternance sooner.
\end{obs}

\subsection{Proper transformations and the proof of a conjecture}
\label{sub:proper}
In 2011 \cite{Tre2011} it has been conjectured that the number of eigenvalues of any tree (with $n$ vertices) smaller than the average degree $d=2- \frac{2}{n}$ is at least $ \lceil \frac{n}{2}\rceil$.  The key element of its recent proof \cite{Jacobs2021} is the concept of proper transformation of a graph. Namely, a proper transformation changes a graph $T$, producing a new graph $T'$, preserving the number $n$ of vertices in such way that the number of eigenvalues above the average degree, denoted by $\sigma(T)$, does not decrease (but could stay the same), that is, $\sigma(T)\leq \sigma(T')$. The main idea in the paper is that, given a tree $T$ we can find a sequence of proper transformations changing $T$ in to a new tree satisfying the conjecture,  and consequently, $T$ will satisfy the conjecture as well. The success of this strategy relies on a careful choice of a few proper transformations that change a given tree $T$ by a finite number of steps into one of four prototype trees $T_0$, $T_1$, $T_2$ and $T_3$ given in Figure \ref{fig:prototype}. It is not a coincidence the fact that the number of prototypes is four.  Indeed, it is a fortuitous consequence of Lemma~\ref{depend   init  cond  to r}, which states  that for $n \geq 8$, if $0 \leq r \leq \lfloor\frac{n}{4}\rfloor$ then $b_{1}(r)<0$, and consequently $b_{2}(r)>0$. As such, a generalized pendant path, i.e. $r$ paths $P_2$ ($2r$ vertices) attached to a path $P_2$ (as Figure~\ref{newseq_r} illustrates), has exactly half of the Laplacian eigenvalues below the average degree (negative vertices) and the other half of the Laplacian eigenvalues above the average degree (positive vertices).
Eventually, we arrive at a tree having two generalized pendant paths with the largest possible $r$ which is $r=\lfloor\frac{n}{4}\rfloor$ or $r=\lfloor\frac{n}{4}\rfloor-1$ we obtain the control of exactly $4r=4 \lfloor\frac{n}{4}\rfloor$ vertices of $n$. The number of remaining vertices to be analyzed are 0,1,2 or 3, according to the congruence of $n \mod 4$. Those are exactly the prototypes trees.
\begin{figure}[H]
  \centering
  \includegraphics[width=12cm]{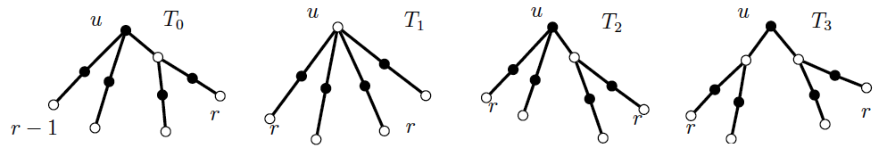}
  \caption{Prototype trees $T_0$, $T_1$, $T_2$ and $T_3$. From \cite{Jacobs2021}, each $T_\alpha, \alpha \in \{0, 1, 2, 3\}$ satisfy $\displaystyle\sigma(T_\alpha) = \left\lfloor\frac{n}{2}\right\rfloor$.} \label{fig:prototype}
\end{figure}

\begin{exemplo}\label{ex:star-up}
In order to obtain one of the prototype trees, one of the goals is to increase the number of $P_2$'s attached to a vertex. Several of these transformations are described and proved to be proper in \cite{Jacobs2021}. As an example we recall the Star-up transform, in which a path is shorten by two vertices, increasing the number of $P_2$'s by one.
   \begin{prop}[\cite{Jacobs2021}, Star-up transform]\label{star-up} Let $u$ be a vertex that is not a leaf of a tree $T$ with $n \geq 8$ vertices. If $u$  has a path  $P_q, ~~ q\geq 2$ connecting $u$ to a starlike vertex that has exactly $0 \leq r \leq \lfloor  \frac{n}{4}\rfloor -1$ pendant $P_2$, and no other pendant path, then the transformation in Figure~\ref{Star-up:fig}~is proper.
\begin{figure}[H]
  \centering
  \definecolor{cqcqcq}{rgb}{0.7529411764705882,0.7529411764705882,0.7529411764705882}
\definecolor{dtsfsf}{rgb}{0.8274509803921568,0.1843137254901961,0.1843137254901961}
\definecolor{ffffff}{rgb}{1,1,1}
\begin{tikzpicture}[line cap=round,line join=round,>=triangle 45,x=1cm,y=1cm, scale=1, every node/.style={scale=0.8}]
\clip(-8.3,-1.22) rectangle (9.64,4.2);
\draw [line width=1pt,color=dtsfsf] (-1.636137486647466,2.074186303595756)-- (-1.16,1.5);
\draw [line width=1pt] (-1.16,1.5)-- (-1.452599533144194,0.6482375879164936);
\draw [line width=1pt] (-1.452599533144194,0.6482375879164936)-- (-1.74,-0.2);
\draw [line width=1pt] (-1.16,1.5)-- (-0.9725771932125596,0.507054546760131);
\draw [line width=1pt] (-0.9725771932125596,0.507054546760131)-- (-0.8607602246167193,-0.3130777393171791);
\draw [line width=1pt] (-1.16,1.5)-- (-0.4643182450496526,0.77530232495722);
\draw [line width=1pt] (-0.4643182450496526,0.77530232495722)-- (0.1,0.18);
\draw (-0.42,0.1) node[anchor=north west] {{\tiny $r$}};
\draw (-2.12,3.5) node[anchor=north west] {{\tiny $u$}};
\draw [line width=1pt,color=dtsfsf] (-1.636137486647466,2.074186303595756)-- (-2.2008696512729182,2.7236282929150244);
\draw [line width=1pt] (-2.2008696512729182,2.7236282929150244)-- (-2.7938384241296434,3.4295434986968374);
\draw (-1.94,3.1) node[anchor=north west] {{\tiny $P_2$}};
\draw (3.14,3.78) node[anchor=north west] {{\tiny $P_{q-2}$}};
\draw [line width=1pt] (4.3562472196987985,2.0720898207028045)-- (4.027420209340036,1.2150874981592914);
\draw [line width=1pt] (4.027420209340036,1.2150874981592914)-- (3.702699214680401,0.4385807717992969);
\draw [line width=1pt] (4.3562472196987985,2.0720898207028045)-- (4.54979746161858,1.03154954465602);
\draw [line width=1pt] (4.54979746161858,1.03154954465602)-- (4.634507286312397,0.2550428182960255);
\draw [line width=1pt] (4.3562472196987985,2.0720898207028045)-- (5.050008976435578,1.3610707627149674);
\draw [line width=1pt] (5.050008976435578,1.3610707627149674)-- (5.616247219698799,0.7520898207028046);
\draw [line width=1pt,color=dtsfsf] (4.3562472196987985,2.0720898207028045)-- (5.3575056400741365,1.8735652021125633);
\draw [line width=1pt,color=dtsfsf] (5.3575056400741365,1.8735652021125633)-- (6.187520739032391,1.7092281422065607);
\draw (5.1,0.68) node[anchor=north west] {{\tiny $r+1$}};
\draw (4.42,2.96) node[anchor=north west] {{\tiny $u$}};
\draw [line width=1pt] (4.3562472196987985,2.0720898207028045)-- (3.7346065819817427,2.7277226011085567);
\draw [->,line width=1pt] (0.3,2.34) -- (2,2.36);
\draw [line width=1pt] (3.7346065819817427,2.7277226011085567)-- (3.068363810764865,3.4281316682852707);
\draw [line width=1pt] (3.068363810764865,3.4281316682852707)-- (2.4192041875279076,4.077291291522226);
\draw [line width=1pt] (-2.7938384241296434,3.4295434986968374)-- (-3.372688892870732,4.107222096247378);
\draw (-3.72,3.5) node[anchor=north west] {{\tiny $P_q$}};
\begin{scriptsize}
\draw [fill=black] (-1.636137486647466,2.074186303595756) circle (2.5pt);
\draw [fill=ffffff] (-1.16,1.5) circle (2.5pt);
\draw [fill=black] (-1.452599533144194,0.6482375879164936) circle (2.5pt);
\draw [fill=ffffff] (-1.74,-0.2) circle (2.5pt);
\draw [fill=black] (-0.9725771932125596,0.507054546760131) circle (2.5pt);
\draw [fill=ffffff] (-0.8607602246167193,-0.3130777393171791) circle (2.5pt);
\draw [fill=black] (-0.4643182450496526,0.77530232495722) circle (2.5pt);
\draw [fill=ffffff] (0.1,0.18) circle (2.5pt);
\draw [fill=cqcqcq] (-2.2008696512729182,2.7236282929150244) circle (2.5pt);
\draw [fill=cqcqcq] (-2.7938384241296434,3.4295434986968374) circle (2.5pt);
\draw [fill=cqcqcq] (4.3562472196987985,2.0720898207028045) circle (2.5pt);
\draw [fill=black] (4.027420209340036,1.2150874981592914) circle (2.5pt);
\draw [fill=ffffff] (3.702699214680401,0.4385807717992969) circle (2.5pt);
\draw [fill=black] (4.54979746161858,1.03154954465602) circle (2.5pt);
\draw [fill=ffffff] (4.634507286312397,0.2550428182960255) circle (2.5pt);
\draw [fill=black] (5.050008976435578,1.3610707627149674) circle (2.5pt);
\draw [fill=ffffff] (5.616247219698799,0.7520898207028046) circle (2.5pt);
\draw [fill=black] (5.3575056400741365,1.8735652021125633) circle (2.5pt);
\draw [fill=ffffff] (6.187520739032391,1.7092281422065607) circle (2.5pt);
\draw [fill=cqcqcq] (3.7346065819817427,2.7277226011085567) circle (2.5pt);
\draw [fill=cqcqcq] (3.068363810764865,3.4281316682852707) circle (2.5pt);
\end{scriptsize}
\end{tikzpicture}
\caption{Star-up transform increasing the number of $P_2$ attached to $u$. Black vertices are positive and white vertices are negative.}\label{Star-up:fig}
\end{figure}
\end{prop}
\end{exemplo}

\begin{coro}\label{double broom 50 percent}
  Consider $T$ a double broom as in the picture below.
\begin{figure}[H]
  \centering
  \includegraphics[width=12cm]{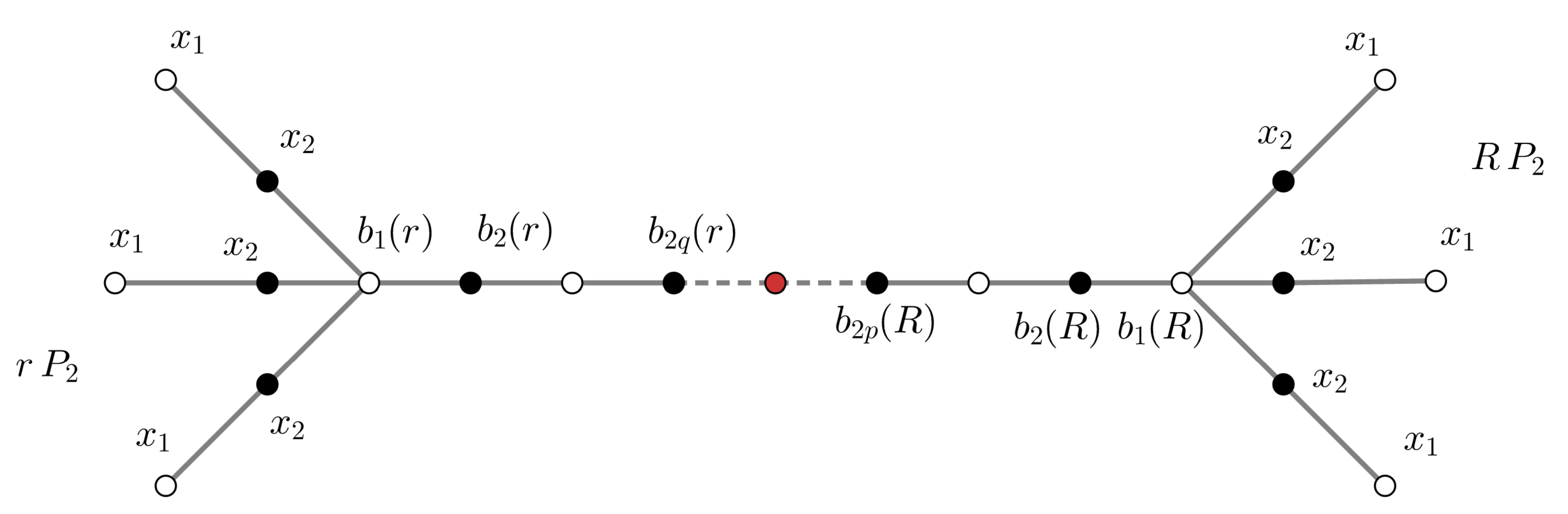}
  \caption{$T$.}\label{broom}
\end{figure}

  If $2q \leq \max\left\{2 \left\lfloor \frac{\pi}{8}(n-2) \right\rfloor - 4(r-1),\;2 \right\}$, $2p \leq \max\left\{2 \left\lfloor \frac{\pi}{8}(n-2) \right\rfloor - 4(R-1),\;2 \right\}$, $r, R < \left\lfloor \frac{n-1}{4} \right\rfloor$ then $\left\lfloor \frac{n}{2} \right\rfloor=\sigma(T) $ or $\left\lfloor \frac{n}{2} \right\rfloor=\sigma(T)+1$.
\end{coro}
\begin{proof}
   It is a direct application of Theorem~\ref{mcsr theorem}. From $2q \leq \max\left\{2 \left\lfloor \frac{\pi}{8}(n-2) \right\rfloor - 4(r-1),\;2 \right\}$, $2p \leq \max\left\{2 \left\lfloor \frac{\pi}{8}(n-2) \right\rfloor - 4(R-1),\;2 \right\}$ we conclude that we have exactly $p+q$ positive vertices and $p+q$ negative vertices. The same is true for each broom producing  exactly $r+R$ positive vertices and $r+R$ negative vertices. Choosing the root as in the Figure~\ref{broom} and applying the J-T algorithm to $T$ we conclude that the only unknown values is the root.

   As $n=2r+2R+2q+2p+1$ we obtain $\left\lfloor \frac{n}{2} \right\rfloor=r+ R+ q+ p=\sigma(T)$ or $\sigma(T)+1$, according  to the signal of $\frac{2 }{n} -\frac{1}{b_{2q}(r)}-\frac{1}{b_{2p}(R)}$ is negative or positive.
\end{proof}


\begin{exemplo}\label{ex:broom_starup} We consider a tree $T$ with $n=19$ vertices given in the Figure~\ref{fig:broom_starup}. Notice that the maximum number of $P_2$'s admitted by our results is $r=\left\lfloor\frac{n}{4}\right\rfloor=4$. Computing the ${\rm mlas}_{r}(19)$, through the exact formula \eqref{eq:estima_old_mcs_r}, we obtain ${\rm mlas}_{1}(19)=12$, ${\rm mlas}_{2}(19)=10$, ${\rm mlas}_{3}(19)=8$ and ${\rm mlas}_{4}(19)=4$.  Since the left side of $T$ has $3P_2$ and is connected to $u$ by $4\leq {\rm mlas}_{3}(19)=8$ vertices and, the right side of $T$ has $2P_2$ and is connected to $u$ by $4\leq {\rm mlas}_{2}(19)=10$ vertices, we can use the Corollary~\ref{double broom 50 percent}, concluding that the sign of each vertex is how we depicted in the figure, that is, 9 negative (eigenvalues below the average degree), 9 positive (eigenvalues above the average degree) and $u$ in red is unknown. If the sign of $u$ were positive we will have $\sigma(T)=10>\left\lfloor\frac{9}{2}\right\rfloor=9$, as prescribed in the conjecture proof from \cite{Jacobs2021}.\\
In order to decide the sign of the remaining vertex, we emulate the argument in \cite{Jacobs2021}. Since $19\equiv 3 \mod{4}$, according with \cite{Jacobs2021}, we must be able to transform it, properly, in to the prototype $T_3$. We will do it using only star-up transform from Proposition~\ref{star-up}:
\begin{itemize}
  \item $T \stackrel{star-up}{\rightarrow} T'$, in the right side, using the root in $v$;
  \item $T' \stackrel{star-up}{\rightarrow} T''$, in the left side, using the root in $v$;
  \item $T'' \stackrel{star-up}{\rightarrow} T'''$, in the right side,  using the root in $u$.
\end{itemize}
We observe that $T'''=T_3$ as we claimed.
\begin{figure}[H]
  \centering
  \includegraphics[width=12cm]{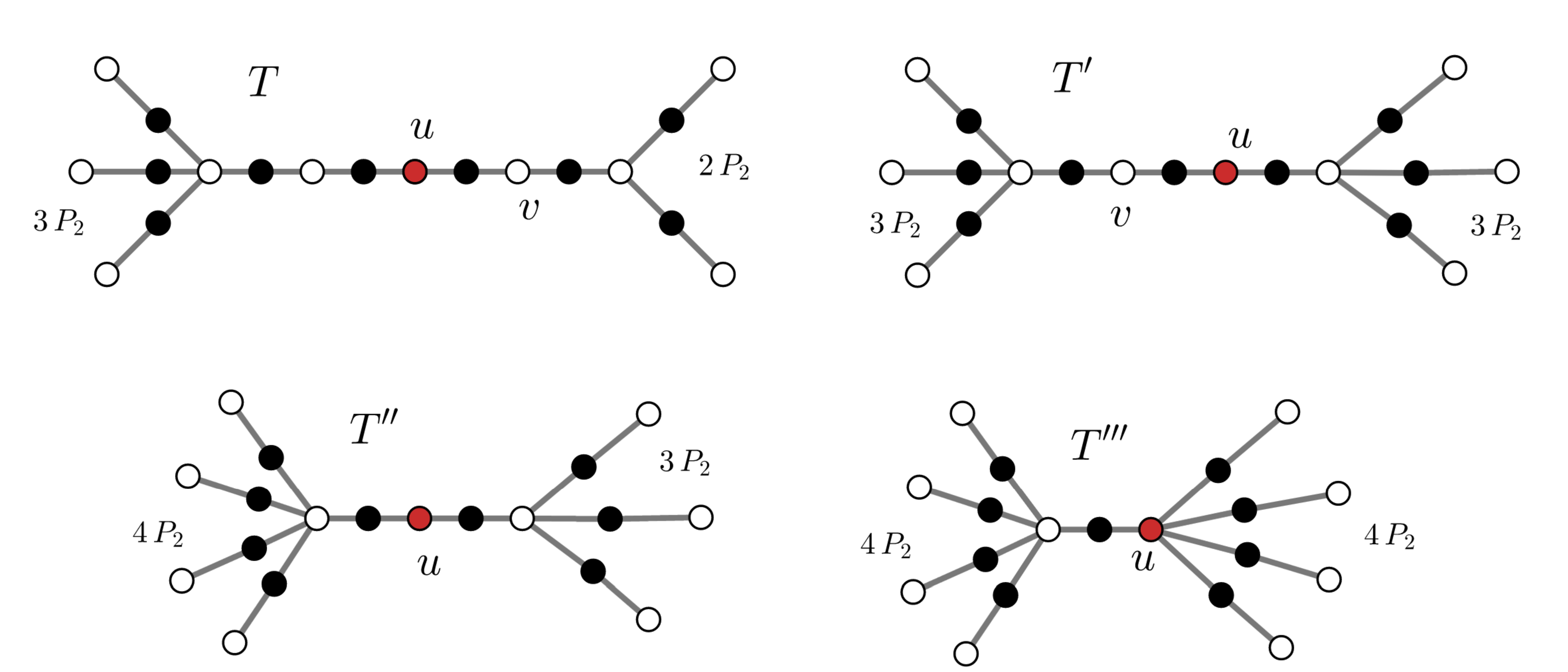}
  \caption{Star-up transform on a broom.} \label{fig:broom_starup}
\end{figure}
Finally, we observe that as the transformations we used are proper we obtain $9 \leq\sigma(T) \leq\sigma(T') \leq\sigma(T'') \leq \sigma(T''')=\sigma(T_3)=9$ concluding that sign of the red vertex, in $T$, is negative. Thus the number of Laplacian eigenvalues below the average degree is $\left\lceil\frac{9}{2}\right\rceil=10$, as expected.
\end{exemplo}

\section{Concluding remarks}\label{sec:conc}
We have studied in this paper a class of recurrence relations given by
$$
  x_{j+1}= \varphi(x_{j}), j \geq 1,
$$
where $\varphi(t)= \alpha + \frac{\gamma}{t}$, for $t\neq 0$,   $\alpha, \gamma \in \mathbb{R}$ are fixed numbers ($\gamma \neq 0$) and $x_{1}$ is a given initial condition. These recurrences appear in the process of diagonalizing certain matrices related to trees. We proposed in this paper an analytical approach to study the behaviour of the solutions of these relations recurrences. This study generalizes some  particular relations that were studied in \cite{Belardo2019,Jacobs2021,oliveira2020spectral,OlivStevTrev}.
The central point about this work is to exhibit a powerful technique which is the passage from natural numbers, representing the vertices of a graph, to arbitrary real numbers, allowing to find solutions of analytical equations which will define intervals of indices where a certain property is true. Through this, we add a new layer of understanding of fine properties of graphs. What makes it possible, is the J-T algorithm, producing rational recurrence formulas which by in turn can be extended to real numbers, whenever we are able to find explicit expressions for them. Our goal is to offer a broad and systematic study of the technical issues involved and show how it works in many important problems of spectral graph theory.
One may argue that trees constitute a particular and small class of graphs. Nevertheless, the number of applications involving trees is sufficient to grant importance to this study. Additionally, it seems that this novel technique has been successful to tackle hard problems where traditional methods have not being able to solve.

It is worth pointing out that the J-T algorithm has been generalized for locating eigenvalues of matrices of a graph based on the parameter treewidth in \cite{Martin2020}. To be more precise, if an order $n$ graph $G$ is given, together with a tree decomposition of width $k$, then it is possible to locate the eigenvalues of $G$ in time $O(k^2n)$. We anticipate that applications of our technique using this algorithm to graphs with small treewidth may be possible.

We would like to finish the paper by suggesting a few problems where this technique may be a tool.
\begin{itemize}
  \item For a given $n$, characterize the trees with $n$ vertices having exactly $\lceil \frac{n}{2}\rceil$ Laplacian eigenvalues smaller than the average degree or, equivalently, $$\sigma(T) = \left\lfloor \frac{n}{2}\right\rfloor.$$
  \item Prove that $P_n$, the path with $n$ vertices, is the only tree minimizing the Laplacian energy among all trees with $n$ vertices.
  \item Find limit points for eigenvalues of graphs.
\end{itemize}

\section*{Acknowledgments}
This research took place while Vilmar Trevisan was visiting the
``Dipartimento di Matematica e Applicazioni'', University of Naples ``Federico II'', Italy.  The authors acknowledge the financial support provided by the hosting University, and by CAPES-Print 88887.467572/2019-00, Brazil. The second author also acknowledges partial support of CNPq grants 409746/2016-9 and 310827/2020-5, and FAPERGS PqG 17/2551-0001.

\end{document}